\documentclass[10pt,A4]{article}
\usepackage{amsmath,amsfonts,amssymb,amsthm}
\usepackage{mathrsfs}
\usepackage{latexsym}
\usepackage[english]{babel}

\def\Cov{{\rm Cov\,}}

\newcommand{\field}[1]{\mathbb{#1}}
\newcommand{\R}{\field{R}}

\newcommand{\Var}{{\rm Var}}

\newcommand{\F}{{\mathscr{F}}}

\newcommand{\B}{{\mathscr B}}

\newcommand{\cS}{{\mathcal S}}
\newcommand{\goto}{{\longrightarrow}}

\def\cS{{\field S}}
\def\E{{\mathbb{ E}}}

\def\P{{\mathbb{P}}}
\def\F{{\mathscr{F}}}

\def\paref#1{(\ref{#1})}
\def\tfrac#1#2{{\textstyle\frac {#1}{#2}}}
\newtheorem{theorem}{Theorem}[section]
\newtheorem{rema}[theorem]{Remark}
\numberwithin{equation}{section}

\newtheorem{cor}[theorem]{Corollary}

\newtheorem{lemma}[theorem]{Lemma}
\newtheorem{prop}[theorem]{Proposition}
\newtheorem{remark}[theorem]{Remark}

\begin{document}

\title{\textbf{\textsc{Stein-Malliavin Approximations for Nonlinear
Functionals of Random Eigenfunctions on ${\mathbb{S}}^{d}$ \thanks{We are deeply grateful
to Igor Wigman for many insightful comments and
suggestions on an earlier draft; our paper exploits several ideas from his
publications, as well as many general results by Ivan Nourdin and Giovanni
Peccati on the Stein-Malliavin approach. Usual disclaimers apply. }
}\ }}
\author{Domenico Marinucci{\textbf{\thanks{%
Research Supported by ERC Grant 277742 \textit{Pascal}}} \text{ } and \text{ } Maurizia Rossi\footnotemark[2]} \\
\textit{Department of Mathematics, University of Rome Tor Vergata,
Italy}
}
\date{}
\maketitle

\begin{abstract}
We investigate Stein-Malliavin approximations for nonlinear functionals of
geometric interest of Gaussian random eigenfunctions on the unit $d$%
-dimensional sphere ${\mathbb{S}}^{d},$ $d\geq 2.$ All our results
are established in the high energy limit, i.e. for eigenfunctions
corresponding to growing eigenvalues. More precisely, we provide
an asymptotic analysis for the variance of random eigenfunctions,
and also establish rates of convergence for various probability
metrics for Hermite subordinated processes, arbitrary polynomials
of finite order and square integral nonlinear transforms; the
latter, for instance, allows to prove a quantitative Central Limit
Theorem for the excursion area. Some related issues were already
considered in the literature for the $2$-dimensional case
${\mathbb{S}}^{2}$; our results are new or improve the existing
bounds even for this special case. Proofs are based on the
asymptotic analysis of moments of all order for Gegenbauer
polynomials, and make extensive use of the recent literature on
so-called fourth-moment theorems by Nourdin and Peccati.

\begin{itemize}
\item \textbf{Keywords and Phrases: }Spherical Harmonics, Gaussian
Eigenfunctions, High Energy Asymptotics, Stein-Malliavin Approximations,
Excursion Area

\item \textbf{AMS Classification: }60G60; 42C10, 60D05, 60B10\smallskip
\end{itemize}
\end{abstract}

\section{Introduction}

The characterization of the asymptotic behaviour (in the high energy limit,
i.e. for eigenfunctions with growing eigenvalues) of geometric functionals
of Gaussian random eigenfunctions on compact manifolds is a topic which has
recently drawn considerable attention. For instance, a growing literature
has focussed on the investigation of the asymptotic behaviour of nodal
lines, i.e. the zero sets of eigenfunctions in some random setups, or the
geometry of nodal domains; in particular, much effort has been devoted to
the $d$-dimensional torus $\mathbb{T}^{d}$ and the unit sphere $\mathbb{S}%
^{d}\subseteq \mathbb{R}^{d+1}$ (see \cite{bogomolnyschmit}, \cite{BGS}%
, \cite{granville}, \cite{khrishnapur}, \cite{Wig1}, \cite{Wig2} e.g.). Many of
these papers have considered the computation of asymptotic variances in the
high energy limit; Central Limit Theorem results have also been established,
for instance for the so-called Defect in the two-dimensional case of the
sphere ${\mathbb{S}}^{2}$ \cite{MaWi12}.

This stream of literature has been largely motivated by applications from
Mathematical Physics. In particular, according to Berry's Universality
conjecture \cite{Berry 1977}, random Gaussian monochromatic waves (similar
to e.g. random Gaussian spherical harmonics) could model deterministic
eigenfunctions on a ``generic" manifold with or without boundary; this
heuristic has strongly motivated the analysis of nodal sets of the former.
On the other hand, it is also well-known that random eigenfunctions are the
Fourier components of square integrable isotropic fields on manifolds. In
view of this and in light of the importance of spherical random fields in
astrophysics and cosmology, the analysis of polynomial transforms or
geometric functionals of random spherical harmonics is a major thread in
these disciplines; these results are used for testing the adequacy of
theoretical models to capture geometric features of observed data (for
instance on Cosmic Microwave Background radiation, see \cite{lewis}, \cite%
{matsubara} or the monograph \cite{MaPeCUP}).

A CLT by itself can often provide little guidance to the
actual distribution of random functionals, as it is only an asymptotic
result with no information on the speed of convergence to the limiting
distribution. More refined results indeed aim at the investigation of the
asymptotic behaviour for various probability metrics, such as Kolmogorov,
Total Variation and Wasserstein distances (to be defined below). In this
respect, a major development in the last few years has been provided by the
so-called \emph{fourth-moments literature, }which is summarized in the
recent monograph \cite{noupebook}. In short, a rapidly growing family of
results is showing how it is possible to establish sharp bounds on
probability distances between multiple stochastic integrals and the standard
Gaussian distribution by means of the analysis of the fourth-moments/fourth
cumulants alone. Such results are currently being generalized in several
directions, including Poisson processes, free probability, random matrices,
Markov subordinators and information theory (see \cite{simon}, \cite{ledoux}%
, \cite{NoPeJFA}, \cite{nourdinpeccatifree}, \cite{swan} e.g.); in
the present paper, we will stick to Gaussian subordinated
circumstances, as described in \S $1.1$.

\subsection{Main results}

\label{main results}

Let us first fix some notation: for any two positive sequence $a_{n},b_{n}$,
we shall write $a_{n}\sim b_{n}$ if $\lim_{n\rightarrow \infty }\frac{a_{n}}{%
b_{n}}=1$ and $a_{n}\ll b_{n}$ or $a_{n}=O(b_{n})$ if the sequence $\frac{%
a_{n}}{b_{n}}$ is bounded. Moreover if $\lim_{n\rightarrow \infty } a_n = \lim_{n\rightarrow \infty } b_n =0$,
then  $a_n = o(b_n)$ if $\lim_{n\rightarrow \infty }\frac{a_{n}}{%
b_{n}}=0$.
Also, we write as usual $dx$ for the Lebesgue
measure on the unit $d$-dimensional sphere ${\mathbb{S}}^{d}\subseteq \R^{d+1}$, so that $%
\int_{{\mathbb{S}}^{d}}\,dx=\mu _{d}$ where $\mu _{d}:=\frac{2\pi ^{\frac{d+1%
}{2}}}{\Gamma \left( \frac{d+1}{2}\right) }$.
The triple $(\Omega, \F, \P)$ shall denote a probability space and $\E$ shall
stand for the expectation w.r.t $\P$; convergence (resp. equality) in law
shall be denoted by $\mathop{\rightarrow}^{\mathcal{L}}$ (resp. $\mathop{=}^{\mathcal{L}}$)
and finally, as usual, $\mathcal N(\mu, \sigma^2)$ shall stand for a Gaussian
random variable with mean $\mu$ and variance $\sigma^2$.

Now let $\Delta _{{\mathbb{S}}^{d}}$ $(d\geq 2)$ denote as usual the
spherical Laplacian operator on  $\mathbb{S}%
^{d}$ and  $\left( Y_{\ell ,m;d}\right) _{\ell ,m}$ the
orthonormal system of (real-valued) spherical harmonics, i.e. for $\ell \in
\mathbb{N}$ the set of eigenfunctions
\begin{equation*}
\Delta _{{\mathbb{S}}^{d}}Y_{\ell ,m;d}=-\ell (\ell +d-1)Y_{\ell ,m;d}\
,\quad m=1,2,\dots ,n_{\ell ;d}\ .
\end{equation*}%
As well-known, the spherical harmonics $\left( Y_{\ell ,m;d}\right)
_{m=1}^{n_{\ell ;d}}$ represent a family of linearly independent homogeneous
polynomials of degree $\ell $ in $d+1$ variables restricted to ${\mathbb{S}}%
^{d}$ of size
\begin{equation*}
n_{\ell ;d}:=\frac{2\ell +d-1}{\ell }{\binom{\ell +d-2}{\ell -1}}\ \sim \
\frac{2}{(d-1)!}\ell ^{d-1}\text{ ,}\quad \text{as}\ \ell \rightarrow
+\infty \ ,
\end{equation*}%
see e.g. \cite{andrews} for further details. It is then customary to
construct, for $\ell \in \mathbb{N}$, the random eigenfunction $T_{\ell }$
on ${\mathbb{S}}^{d}$ by taking
\begin{equation}\label{Telle}
T_{\ell }(x):=\sum_{m=1}^{n_{\ell ;d}}a_{\ell ,m}Y_{\ell ,m;d}(x)\ ,\quad
x\in {\mathbb{S}}^{d}\ ,
\end{equation}%
with the coefficients $\left( a_{\ell ,m}\right) _{m=1}^{n_{\ell ;d}}$
Gaussian i.i.d. random variables, satisfying the relation
\begin{equation*}
{\mathbb{E}}[a_{\ell ,m}a_{\ell ,m^{\prime }}]=\frac{\mu _{d}}{n_{\ell ;d}}%
\delta _{m}^{m^{\prime }}\ ,
\end{equation*}%
where
$\delta _{a}^{b}$ denotes the Kronecker delta function and $\mu_d=\frac{2\pi ^{\frac{d+1%
}{2}}}{\Gamma \left( \frac{d+1}{2}\right) }$ the hypersurface volume of the $d$-dimensional unit
sphere as above.

It is then readily checked that $\left( T_{\ell }\right) _{\ell \in \mathbb{N%
}}$ represents a sequence of isotropic, mean-zero Gaussian random fields on $%
{\mathbb{S}}^{d}$, that is, for every fixed $\ell $ we have a collection of
random variables $\left( T_{\ell }(x)\right) _{x\in {\mathbb{S}}^{d}}$
indexed by the points of ${\mathbb{S}}^{d}$, such that the map
\begin{equation*}
T_{\ell }:\,\Omega \times {\mathbb{S}}^{d}{\longrightarrow }\mathbb{R}\text{
};\qquad (\omega ,x)\mapsto T_{\ell }(\omega ,x)
\end{equation*}%
is ${\mathscr{F}}\otimes {\mathscr B}({\mathbb{S}}^{d})$-measurable, where ${%
\mathscr B}({\mathbb{S}}^{d})$ denotes the Borel $\sigma $-field of ${%
\mathbb{S}}^{d}$. The isotropy of $T_{\ell}$ means that the probability laws of
the two random fields $T_{\ell }(\cdot )$ and $T_{\ell }^{g}(\cdot
):=T_{\ell }(g\,\cdot )$ are equal for every $g\in SO(d+1).$

It is also well-known that every Gaussian and isotropic random field $T$ on $%
\mathbb{S}^{d}$ is necessarily mean-square continuous (indeed this statement
holds for every isotropic and finite variance random field on a homogeneous
space of a compact group - see \cite{MaPe}) and satisfies in the  $%
L^{2}(\Omega \times {\mathbb{S}}^{d})$-sense the spectral representation
(see \cite{mal2013}, \cite{MaPeCUP} and also \cite{adlertaylor}, \cite%
{baldirossi})
\begin{equation*}
T(x)=\sum_{\ell =1}^{\infty }c_{\ell }T_{\ell }(x)\ ,\quad x\in {\mathbb{S}}%
^{d}\text{ ,}
\end{equation*}%
where $\mathbb{E}\left[ T^{2}\right] =\sum_{\ell =1}^{\infty }c_{\ell
}^{2}<\infty $; hence the spherical Gaussian eigenfunctions $\left( T_{\ell
}\right) _{\ell \in \mathbb{N}}$ can be viewed as the Fourier components of
the field $T$ (note that w.l.o.g. we are implicitly assuming that $T$ is
centred). Equivalently these random  eigenfunctions \paref{Telle}
could be defined
by their covariance function, which equals
\begin{equation}
{\mathbb{E}}[T_{\ell }(x)T_{\ell }(y)]=G_{\ell ;d}(\cos d(x,y))\ ,\quad
x,y\in {\mathbb{S}}^{d}\ .  \label{defT}
\end{equation}%
Here and in the sequel, $d(x,y)$ is the spherical distance between $x,y\in
\mathbb{S}^{d}$, and $G_{\ell ;d}:[-1,1]{\longrightarrow }\mathbb{R}$ is the
$\ell $-th Gegenbauer polynomial, i.e. $G_{\ell ;d}\equiv P_{\ell }^{(\frac{d%
}{2}-1,\frac{d}{2}-1)},$ where $P_{\ell }^{(\alpha ,\beta )}$ are the Jacobi
polynomials; as a special case, for $d=2$, it equals $G_{\ell ;2}\equiv
P_{\ell },$ the degree-$\ell $ Legendre polynomial. Throughout this paper,
we normalize so that $G_{\ell ;d}(1)=1$. Recall that the Jacobi polynomials $%
P_{\ell }^{(\alpha ,\beta )}$ are orthogonal on the interval $[-1,1]$ with
respect to the weight function $w(t)=(1-t)^{\alpha }(1+t)^{\beta }$ and
satisfy
\begin{equation*}
P_{\ell }^{(\alpha ,\beta )}(1)=\left(
\begin{array}{c}
\ell +\alpha \\
\ell%
\end{array}%
\right) \text{ ,}
\end{equation*}%
see \cite{szego} for more details.

The main purpose of this paper is to investigate quantitative CLTs
 for nonlinear functionals of Gaussian spherical eigenfunctions on ${%
\mathbb{S}}^{d}$. For $d=2$ this issue was addressed in \cite{MaWi12}; our
first aim is to extend their results to arbitrary dimensions and study the
asymptotic behavior, as $\ell \rightarrow \infty $, of the random variables $%
h_{\ell ;q,d}$ defined for $\ell =1,2,\dots $ and $q=0,1,\dots $ as
\begin{equation}
h_{\ell ;q,d}=\int_{{\mathbb{S}}^{d}}H_{q}(T_{\ell }(x))\,dx\ ,  \label{hq}
\end{equation}%
where $H_{q}$ represent the family of Hermite polynomials (\cite{noupebook},
\cite{peccatitaqqu}). The latter are defined as usual by $H_{0}\equiv 1$ and
for $q=1,2,\dots $
\begin{equation}
H_{q}(t)=(-1)^{q}\mathrm{e}^{\frac{t^{2}}{2}}\frac{d^{q}}{dt^{q}}\mathrm{e}%
^{-\frac{t^{2}}{2}}\ ,\ t\in \mathbb{R}\ .  \label{hermite}
\end{equation}
Note that, for all $d$
\begin{equation*}
h_{\ell ;0,d}=\mu _{d}\ ,\qquad h_{\ell ;1,d}=0
\end{equation*}%
a.s., and therefore it is enough to restrict our discussion to $q\geq 2$.
Moreover ${\mathbb{E}}[h_{\ell; q, d}]=0$ and
\begin{equation}
\mathrm{Var}[h_{\ell ;q,d}]=q!\mu _{d}\mu _{d-1}\int_{0}^{\pi }G_{\ell
;d}(\cos \vartheta )^{q}(\sin \vartheta )^{d-1}\,d\vartheta  \label{int var}
\end{equation}%
(see \S 3 for more details). Gegenbauer polynomials satisfy the symmetry
relationships
\begin{equation*}
G_{\ell ;d}(t)=(-1)^{\ell }G_{\ell ;d}(-t)\ ,
\end{equation*}%
whence the r.h.s. integral in (\ref{int var}) vanishes identically when both
$\ell $ and $q$ are odd; hence in these cases $h_{\ell ;q,d}= 0$ a.s. For
the remaining cases we have
\begin{equation}\label{int var doppio}
\mathrm{Var}[h_{\ell ;q,d}]=2q!\mu _{d}\mu _{d-1}\int_{0}^{\frac{\pi }{2}%
}G_{\ell ;d}(\cos \vartheta )^{q}(\sin \vartheta )^{d-1}\,d\vartheta \ .
\end{equation}%
Our first result, given in \S 3, is an upper bound for these variances,
asymptotic for $\ell \rightarrow \infty $.

\begin{prop}
\label{varianza} As $\ell \rightarrow \infty ,$ for $q,d\ge 3$,
\begin{equation}  \label{int1}
\int_{0}^{\frac{\pi }{2}}G_{\ell ;d}(\cos \vartheta )^{q}(\sin \vartheta
)^{d-1}\,d\vartheta =\frac{c_{q;d}}{\ell ^{d}}(1+o_{q;d}(1)).
\end{equation}
The constants $c_{q;d}$ are given by the formula
\begin{equation}
c_{q;d}=\left(2^{\frac{d}{2} - 1}\left (\frac{d}2-1 \right)!\right)^q\int_{0}^{+\infty
}J_{\frac{d}{2}-1}(\psi )^{q}\psi ^{-q\left( {\textstyle\frac{d}{2}}%
-1\right) +d-1}d\psi\ ,  \label{cq}
\end{equation}%
where $J_{\frac{d}{2}-1}$ is the Bessel function
of order $\frac{d}{2}-1$. The r.h.s. integral in (\ref{cq}) is absolutely convergent for any
pair $(d,q)\neq (3,3)$ and conditionally convergent for $d=q=3$.
\end{prop}

It is well known that for $d\geq 2$, the second moment of the Gegenbauer
polynomials is given by
\begin{equation}
\int_{0}^{\pi }G_{\ell ;d}(\cos \vartheta )^{2}(\sin \vartheta
)^{d-1}\,d\vartheta =\frac{\mu _{d}}{\mu _{d-1}\,n_{\ell ;d}}\ ,
\label{momento 2}
\end{equation}%
whence
\begin{equation}
\mathrm{Var}[h_{\ell ;2,d}]=2\frac{\mu _{d}^{2}}{n_{\ell ;d}}\ \sim\  4\mu_d \mu_{d-1}\frac{c_{2;d}}{\ell ^{d-1}}\ ,\qquad \text{as}\ \ell \rightarrow
+\infty \ , \label{q=2}
\end{equation}%
where $c_{2;d} :=
\frac{(d-1)!\mu _{d}}{4 \mu_{d-1}}$.
For $d=2$ and every $q$, the asymptotic behaviour of these integrals was resolved in \cite%
{MaWi2}. In particular, it was shown that for $q=3$ or $q\geq 5$
\begin{equation}
\mathrm{Var}[h_{\ell ;q,2}]=(4\pi )^{2}q!\int_{0}^{\frac{\pi}{2}}P_{\ell }(\cos
\vartheta )^{q}\sin \vartheta \,d\vartheta =(4\pi )^{2}q!\frac{c_{q;2}}{\ell
^{2}}(1+o_{q}(1))\ ,  \label{int2}
\end{equation}%
where
\begin{equation}
c_{q;2}=\int_{0}^{+\infty }J_{0}(\psi )^{q}\psi \,d\psi \ ,  \label{cq2}
\end{equation}%
$J_{0}$ being the Bessel function of order $0$ and the above integral being
absolutely convergent for $q\ge 5$ and conditionally convergent for $%
q=3$. On the other hand, for $q=4$, as $\ell \rightarrow \infty $,
\begin{equation}
\mathrm{Var}[h_{\ell;4,2}]\ \sim \ 24^{2}\frac{\text{log}\ell }{\ell ^{2}}\
.  \label{q=4d=2}
\end{equation}%
Clearly for any $d,q\geq 2$, the constants $c_{q;d}$ are
nonnegative and it is obvious that $c_{q;d}>0$ for all even $q$.
We conjecture that this strict inequality holds for every $(d,q)$,
but leave this issue as an open question for future research;
also, in view of the previous discussion on the symmetry
properties of Gegenbauer polynomials, to simplify the discussion
in the sequel we restrict ourselves to even multipoles $\ell $.

In this paper we first establish quantitative CLTs for $%
h_{\ell ;q,d}$ (see \S 4) and then for other nonlinear functionals of
geometric interest (see \S 5,6). Our results are new for $d\geq 3;$ for $d=2$
we improve the existing bounds on probability metrics that were readily
established \cite{MaWi12}, and also extend this analysis to non-Hermite
polynomials, and establish Breuer-Major like results with surprisingly fast
convergence rates for generic nonlinear functionals, including e.g. the area
of excursion sets (see the discussion below for more details).

To formulate our results we need to introduce some more notation. Denote the
usual Kolmogorov $d_{K}$, Total Variation $d_{TV}$ and Wasserstein $d_{W}$
distances between random variables $Z,N$:
\begin{eqnarray*}
d_{K}(Z,N) &=&\sup_{z\in \mathbb{R}}\left\vert {\mathbb{P}}(Z\leq z)-{%
\mathbb{P}}(N\leq z)\right\vert \text{ ,}  \label{prob distance} \\
d_{TV}(Z,N) &=&\sup_{A\in {\B}(\mathbb{R})}\left\vert {\mathbb{P}}%
(Z\in A)-{\mathbb{P}}(N\in A)\right\vert \text{ ,} \\
d_{W}(Z,N) &=&\sup_{h\in \text{Lip}(1)}\left\vert {\mathbb{E}}[h(Z)]-{%
\mathbb{E}}[h(N)]\right\vert \text{ ,}
\end{eqnarray*}%
where ${\B}(\mathbb{R})$ denotes the Borel $\sigma $-field of $%
\mathbb{R}$ and Lip$(1)$ the set of Lipschitz functions whose Lipschitz
constant equals $1$. We shall prove the following result.

\begin{theorem}
\label{teo1} For all $d,q=2,3,\dots $,
$d_{\mathcal{D}}=d_{TV},d_{W},d_{K}$ we have%
\begin{equation*}
d_{\mathcal{D}}\left( \frac{h_{2\ell ;q,d}}{\sqrt{Var[h_{2\ell ;q,d}]}},%
\mathcal{N}(0,1)\right) =O(R(\ell ;q,d))\ ,  \label{Dbound}
\end{equation*}%
where for $d=2$
\begin{equation}
R(\ell ;q,2)=%
\begin{cases}
\ell ^{-\frac{1}{2}}\qquad & \quad q=2,3, \\
(\log \ell )^{-1}\qquad & \quad q=4, \\
(\log \ell ){\ell ^{-\frac{1}{4}}}\qquad & \quad q=5,6, \\
{\ell ^{-\frac{1}{4}}}\qquad & \quad q\geq 7;%
\end{cases}
\label{rate2}
\end{equation}%
and for $d=3,4,\dots $
\begin{equation}\label{rateq}
R(\ell ;q,d)=%
\begin{cases}
\ell ^{-\left( \frac{d-1}{2}\right) }\quad & \quad q=2, \\
\ell ^{-\left( \frac{d-5}{4}\right) }\quad & \quad q=3, \\
\ell ^{-\left( \frac{d-3}{4}\right) }\qquad & \quad q=4, \\
\ell ^{-\left( \frac{d-1}{4}\right) }\qquad & \quad q\geq 5.%
\end{cases}%
\end{equation}
\end{theorem}
The following corollary is hence immediate.

\begin{cor}
\label{cor1} For all $q$ such that $(d,q)\neq (3,3), (3,4),(4,3),(5,3)$ and $%
c_{q;d}>0$, $d=2,3,\dots $,
\begin{equation}
\frac{h_{2\ell ;q,d}}{\sqrt{Var[h_{2\ell ;q,d}]}}\mathop{\goto}^{\mathcal{L}}%
\mathcal{N}(0,1)\ ,\qquad \text{as}\ \ell \rightarrow +\infty \ .
\label{convergenza hq}
\end{equation}
\end{cor}

\begin{remark}
\textrm{For $d=2$, the CLT in (\ref{convergenza hq}) was already provided by
\cite{MaWi12}; nevertheless Theorem \ref{teo1} improves the existing bounds
on the speed of convergence to the asymptotic Gaussian distribution. More
precisely, for $d=2,q=2,3,4$ the same rate of convergence as in (\ref{rate2}%
) was given in their Proposition 3.4; however for arbitrary $q$ the total
variation rate was only shown to satisfy (up to logarithmic terms) $%
d_{TV}=O(\ell ^{-\delta _{q}}),$ where $\delta _{4}=\frac{1}{10},$ $\delta
_{5}=\frac{1}{7},$ and $\delta _{q}=\frac{q-6}{4q-6}<\frac{1}{4}$ for $q\geq
7$.}
\end{remark}

\begin{remark}
\textrm{The cases not included in Corollary \ref{cor1} correspond to the
pairs where $q=4$ and $d=3$, or $q=3$ and $d=3,4,5$; in these circumstances
the bounds we establish on fourth-order cumulants are not sufficient to
ensure that the CLT holds. Most probably, these four
special cases can be dealt with ad hoc arguments based on the explicit
evaluations of multiple integrals of spherical harmonics by means of
so-called Clebsch-Gordan coefficients, following the steps of Lemma 3.3 in
\cite{MaWi12}, see also \cite{M2008}, \cite{MaPeCUP}. Such computations,
however, seem of limited interest for the present paper, and we therefore
omit the investigation of these special cases for brevity's sake.}
\end{remark}
The random variables $h_{\ell;q,d}$ defined in (\ref{hq}) are the basic
building blocks for the analysis of any square integrable nonlinear transforms
of Gaussian spherical eigenfunctions on ${\mathbb{S}}^{d}$. Indeed, let us
consider generic polynomial functionals of the form
\begin{equation}
Z_{\ell }=\sum_{q=0}^{Q}b_{q}\int_{{\mathbb{S}}^{d}}
T_{\ell
}(x)^{q}\,dx\ ,\qquad Q\in \mathbb{N},\text{ }b_{q}\in \mathbb{R},
\label{Z}
\end{equation}%
which include, for instance, the so-called polyspectra of isotropic random
fields defined on ${\mathbb{S}}^{d}$. Note
\begin{equation}
Z_{\ell }=\sum_{q=0}^{Q}\beta _{q}h_{2\ell ;q,d}  \label{Z2}
\end{equation}%
for some $\beta _{q}\in \mathbb{R}$. It is easy to establish CLTs
 for generic polynomials (\ref{Z2}) from convergence results on
$ h_{2\ell ;q,d}$, see e.g. \cite{peccatitudor}. It is more
difficult to investigate the speed of convergence in the CLT in terms of the
probability metrics we introduced earlier; indeed, in \S 5 we establish the
following.

\begin{theorem}
\label{corollario1} As $\ell \rightarrow \infty ,$%
\begin{equation*}
d_{\mathcal{D}}\left( \frac{Z_{\ell }-{\mathbb{E}}[Z_{\ell }]}{\sqrt{\mathrm{%
Var}[Z_{\ell }]}},\mathcal{N}(0,1)\right) =O(R(Z_{\ell };d))\text{ ,}
\end{equation*}%
where $d_{\mathcal{D}}=d_{TV},d_{W},d_{K}$ and for $d=2,3,\dots $
\begin{equation*}
R(Z_{\ell };d)=%
\begin{cases}
\ell ^{-\left( \frac{d-1}{2}\right) }\qquad & \text{if}\quad \beta _{2}\neq
0\ , \\
\max_{q=3,\dots ,Q\,:\,\beta _{q},c_{q;d}\neq 0}R(\ell ;q,d)\qquad & \text{if%
}\quad \beta _{2}=0\ .%
\end{cases}%
\end{equation*}
\end{theorem}

All the results as above can be summarized as follows: for polynomials of
Hermite rank $2$ (i.e. their projection against $H_{2}(T_{\ell })$ $\beta
_{2}\neq 0,$ does not vanish), the asymptotic behaviour of $Z_{\ell }$ is
dominated by the term $h_{\ell ;2,d},$ whose variance is of order $\ell
^{-d+1}$ rather than $O(\ell ^{-d})$ as for the other terms. On the other
hand, when $\beta _{2}=0,$ the convergence rate to the asymptotic Gaussian
distribution for a generic polynomial is the slowest among the rates for the
Hermite components into which $Z_{\ell }$ can be decomposed.

The fact that the bound for generic polynomials is of the same order as for
the Hermite case (and not slower) is indeed rather unexpected; it can be
shown to be due to the cancellation of some cross-product terms, which are
dominating in the general Nourdin-Peccati framework, while they vanish in
the framework of spherical eigenfunctions of arbitrary dimension (see (\ref%
{eq=0}) and Remark \ref{rem0}). An inspection of our proof will reveal that
this result is a by-product of the orthogonality of eigenfunctions
corresponding to different eigenvalues; it is plausible that similar ideas
may be exploited in many related circumstances, for instance random
eigenfunction on generic compact manifolds.

\vspace{5mm}

Theorem \ref{corollario1} shows that the asymptotic behaviour of arbitrary
polynomials of Hermite rank $2$ is of particularly simple nature. Our result
below will show that this feature holds in much greater generality, at least
as far as the Wasserstein distance is concerned. Indeed, we shall consider
the case of functionals of the form
\begin{equation}
S_{\ell }(M)=\int_{{\mathbb{S}}^{d}}M(T_{\ell }(x))\,dx\ ,  \label{S}
\end{equation}%
where $M:\mathbb{R}\rightarrow \mathbb{R}$ is any square integrable,
measurable nonlinear function. It is well known that for such transforms the
following expansion holds in $L^{2}(\Omega )$-sense
\begin{equation}
M(T_{\ell })=\sum_{q=0}^{\infty }\frac{J_{q}(M)}{q!}H_{q}(T_{\ell }),\ \ \ {%
\mathbb{E}}[M(T_{\ell })^{2}]<\infty ,\ \ \ J_{q}(M):={\mathbb{E}}[M(T_{\ell
})H_{q}(T_{\ell })]\ .  \label{exp}
\end{equation}%
Therefore the asymptotic analysis, as $\ell \rightarrow \infty $, of $%
S_{\ell }(M)$ in (\ref{S}) directly follows from the Gaussian approximation
for $h_{\ell ;q,d}$ and their polynomial transforms $Z_{\ell }$. More
precisely, in \S 6 we prove the following result.

\begin{theorem}
\label{general} For functions $M$ in (\ref{S}) such that $\mathbb{E}\left[
M(Z)H_{2}(Z)\right] =$ $J_{2}(M)\neq 0$, we have
\begin{equation}
d_{W}\left( \frac{S_{2\ell }(M)-{\mathbb{E}}[S_{2\ell }(M)]}{\sqrt{%
Var[S_{2\ell }(M)]}},\mathcal{N}(0,1)\right) =O(\ell ^{-\frac{1}{2}}) \
,\qquad \text{as}\ \ell \rightarrow \infty\ ,  \label{sun2}
\end{equation}
in particular
\begin{equation}
\frac{S_{2\ell }(M)-{\mathbb{E}}[S_{2\ell }(M)]}{\sqrt{\mathrm{Var}[S_{2\ell
}(M)]}}\mathop{\goto}^{\mathcal{L}}\mathcal{N}(0,1)\ .  \label{sun1}
\end{equation}
\end{theorem}

Theorem \ref{general} provides a Breuer-Major like result on nonlinear
functionals, in the high-frequency limit (compare for instance \cite%
{nourdinpodolskij}). While the CLT in \eqref{sun1} is somewhat expected, the
square-root speed of convergence \eqref{sun2} to the limiting distribution
 may be considered quite remarkable; it is mainly due to some specific
features in the chaos expansion of random eigenfunctions, which is dominated
by a single term at $q=2.$ Note that the function $M$ need not be smooth in
any meaningful sense; indeed our main motivating rationale here is the
analysis of the asymptotic behaviour of the empirical measure for excursion
sets, where $M(\cdot )=M_{z}(\cdot )=\mathbb{I}(\cdot \leq z)$ is the
indicator function of the interval $(-\infty, z]$. Therefore, in words, $%
S_{\ell }(z):=S_{\ell }(M_{z})$ is the (random) measure of an excursion set,
i.e. $T_{\ell }$ lies above a given level $z\in \mathbb{R}$; an application
of Theorem \ref{general} yields a quantitative CLT for $%
S_{\ell }(z)$, $z\neq 0$.

\section{Background}

\label{background}

In a number of recent papers summarized in the monograph \cite{noupebook}, a
beautiful connection has been established between Malliavin calculus and the
so-called Stein method to prove Berry-Esseen bounds and quantitative CLTs
 on functionals of Gaussian subordinated random fields. In
this section, we first briefly review some notation and the main results in
this area, which we shall deeply exploit in the sequel of the paper.

\subsection{Stein-Malliavin Normal approximations}

Let us consider the measure space $(X,\mathcal{X},\mu )$, where $X$ is a
Polish space, $\mathcal{X}$ is the $\sigma $-field on $X$ and $\mu $ is a
positive, $\sigma $-finite and non-atomic measure on $(X,\mathcal{X})$.
Denote $H=L^{2}(X,\mathcal{X},\mu )$ the real (separable) Hilbert space of
square integrable functions on $X$ w.r.t. $\mu $, with inner product $%
\langle f,g\rangle _{H}=\int_{X}f(x)g(x)\,d\mu (x)$. Let us recall the
construction of an isonormal Gaussian field on $H$. First consider a
Gaussian white noise on $X$, i.e. a centered Gaussian family $W$
\begin{equation*}
W=\{W(A):A\in \mathcal{X},\mu (A)<+\infty \}
\end{equation*}%
such that for $A,B\in \mathcal{X}$ of finite measure, we have
\begin{equation*}
{\mathbb{\ E}}[W(A)W(B)]=\int_{X}\mathbb{I}({A\cap B})\,d\mu \text{ .}
\end{equation*}%
We define a Gaussian random field $T$ on $H$ as follows. For each $f\in H$,
let
\begin{equation}
T(f)=\int_{X}f(x)\,dW(x)\ ,  \label{isonormal}
\end{equation}%
i.e. the Wiener-Ito integral of $f$ with respect to $W$. The random field $T$ is the
isonormal Gaussian field on $H$; indeed
\begin{equation*}
\mathrm{Cov\,}(T(f),T(g))=\langle f,g\rangle _{H}\ .
\end{equation*}%
Let us recall now the notion of Wiener chaoses. Define the space of
constants $C_{0}:=\mathbb{R}\subseteq L^{2}(\Omega )$, and for $q\geq 1$,
let $C_{q}$ be the closure in $L^{2}(\Omega )$ of the linear subspace
generated by random variables of the form
\begin{equation*}
H_{q}(T(f))\ ,\qquad f\in H,\ \Vert f\Vert _{H}=1\ ,
\end{equation*}%
where $H_{q}$ is the $q$-th Hermite polynomial (\ref{hermite}). $C_{q}$ is
called the $q$-th Wiener chaos. The following, well-known property will be
useful in the sequel: let $Z_{1},Z_{2}\sim \mathcal{N}(0,1)$ be jointly
Gaussian; then, for all $q_{1},q_{2}\geq 0$
\begin{equation}  \label{hermite orto}
{\mathbb{\ E}}[H_{q_{1}}(Z_{1})H_{q_{2}}(Z_{2})]=q_{1}!\,{\mathbb{\ E}}%
[Z_{1}Z_{2}]^{q_{1}}\,\delta _{q_{2}}^{q_{1}}\ .
\end{equation}
Moreover the following chaotic Wiener-Ito expansion holds:
\begin{equation*}
L^{2}(\Omega )=\bigoplus_{q=0}^{+\infty }C_{q}\ ,
\end{equation*}%
the above sum being orthogonal from (\ref{hermite orto}).
Equivalently, each random variable $F\in L^{2}(\Omega )$ admits a unique
decomposition in the $L^{2}(\Omega)$-sense of the form
\begin{equation}
F=\sum_{q=0}^{\infty }J_{q}(F)\ ,  \label{chaos exp}
\end{equation}%
where $J_{q}:L^{2}(\Omega ){\longrightarrow }C_{q}$ is the orthogonal
projection operator. Remark that $J_{0}(F)={\mathbb{\ E}}[F]$.

We denote by $H^{\otimes q}$ and $H^{\odot q}$ the $q$-th tensor product and
the $q$-th symmetric tensor product of $H$ respectively. In particular $%
H^{\otimes q} = L^2(X^{q}, \mathcal{X}^{q},\mu ^{q})$ and $H^{\odot
q}=L^2_s(X^{q}, \mathcal{X}^{q},\mu ^{q})$ where by $L^2_s$ we mean the
square integrable and symmetric functions. Note that for $(x_1,x_2,\dots,
x_q)\in X^q$ and $f\in H$, we have
\begin{equation*}
f^{\otimes q}(x_1,x_2,\dots,x_q)=f(x_1)f(x_2)\dots f(x_q)\ .
\end{equation*}
Now for $q\ge 1$ define the map $I_{q}$ as
\begin{equation}
I_{q}(f^{\otimes q}):=\,H_{q}(T(f))\ ,\qquad f\in H\ ,  \label{isometria}
\end{equation}%
which can be extended to a linear isometry between $H^{\odot q}$ equipped
with the modified norm $\sqrt{q!}\Vert \cdot \Vert _{H^{\odot q}}$ and the $%
q $-th Wiener chaos $C_{q}$. Moreover for $q=0$, set $I_{0}(c)=c\in \mathbb{R%
}$. Under the new notation the equality \eqref{chaos exp} becomes
\begin{equation}
F=\sum_{q=0}^{\infty }I_{q}(f_{q})\ ,  \label{chaos exp2}
\end{equation}%
where $f_{0}={\mathbb{\ E}}[F]$ and for $q\ge 1$, the kernels $f_{q}\in
H^{\odot q}$ are uniquely determined.

In our case, it is well known that for $h\in H^{\odot q}$, $I_q(h)$
coincides with the multiple Wiener-Ito integral of $h$ with respect to the
Gaussian measure $W$, i.e.
\begin{equation}  \label{int multiplo}
I_q(h)= \int_{X^q} h(x_1,x_2,\dots x_q)\,dW(x_1) dW(x_2)\dots dW(x_q)
\end{equation}
and, in words, $F$ in (\ref{chaos exp2}) can be seen as a series of
(multiple) stochastic integrals.

For every $p,q\ge 1$, $f\in H^{\otimes p}, g\in H^{\otimes q}$ and $%
r=1,2,\dots, p\wedge q$, the so-called \emph{contraction} of $f$ and $g$ of
order $r$ is the element $f\otimes _{r}g\in H^{\otimes p+q-2r}$ defined as
\begin{equation}\label{contrazione}
\begin{split}
(f\otimes _{r}g)&(x_{1},\dots,x_{p+q-2r})=\\
=\int_{X^{r}}f(x_{1},\dots,x_{p-r},y_{1},\dots,y_{r})
g(x_{p-r+1}&,\dots,x_{p+q-2r},y_{1},\dots,y_{r})\,d\mu(y_{1})\dots d\mu
(y_{r})\text{ .}
\end{split}
\end{equation}
For $p=q=r$, we have $f\otimes _{r}g = \langle f,g \rangle_{H^{\otimes_r}}$
and for $r=0$, $f\otimes _{0}g = f\otimes g$. Denote by $f\widetilde \otimes
_{r}g$ the canonical symmetrization of $f\otimes_r g$. The following
multiplication formula is well-known; for $p,q=1,2,\dots$, $f\in H^{\odot
p}, g\in H^{\odot q}$, we have
\begin{equation*}
I_p(f)I_q(g)=\sum_{r=0}^{p\wedge q} r! {\binom{p }{r}} {\binom{q }{r}}%
I_{p+q-2r}(f\widetilde \otimes _{r}g)\ .
\end{equation*}
We now briefly recall some basic Malliavin calculus formulas for this
setting. For $q,r\ge 1$, the $r$-th Malliavin derivative of a random
variable $F=I_{q}(f)\in C_{q}$ where $f\in H^{\odot q}$, can be identified as
the element $D^{r}F:\Omega \rightarrow H^{\odot r}$ given by
\begin{equation}
D^{r}F=\frac{q!}{(q-r)!}I_{q-r}(f)\text{ ,}  \label{marra2}
\end{equation}%
for $r\leq q$, and $D^{r}F=0$ for $r>q$. So that, the $r$-th Malliavin
derivative of the random variable $F$ in (\ref{chaos exp2}) could be written
as
\begin{equation*}
D^{r}F=\sum_{q=r}^{+\infty }\frac{q!}{(q-r)!}I_{q-r}(f_{q})\ .
\end{equation*}%
For simplicity of notation, we shall write $D$ instead of $D^{1}$. We say
that $F$ as in (\ref{chaos exp2}) belongs to $\mathbb{D}^{r,q}$ if
\begin{equation*}
\Vert F\Vert _{\mathbb{D}^{r,q}}:=\left( {\mathbb{\ E}}[|F|^{q}]+\dots {%
\mathbb{\ E}}[\Vert D^{r}F\Vert _{H^{\odot r}}^{q}]\right) ^{\frac{1}{q}%
}<+\infty \ ;
\end{equation*}%
it is easy to check that $F\in \mathbb{D}^{1,2}$ if and only if
\begin{equation*}
{\mathbb{\ E}}[\Vert DF\Vert _{H}^{2}]=\sum_{q=1}^{\infty }q\Vert
J_{q}(F)\Vert _{L^{2}(\Omega )}^{2}<+\infty \ .
\end{equation*}
We need to introduce also the generator of the Ornstein-Uhlenbeck semigroup,
defined as
\begin{equation*}
L=-\sum_{q=0}^{\infty }qJ_{q}\ ,
\end{equation*}%
where $J_{q}$ is the orthogonal projection operator on $C_{q}$, as in (\ref%
{chaos exp}). The domain of $L$ is $\mathbb{D}^{2,2}$, equivalently the
space of Gaussian subordinated random variables $F$ such that
\begin{equation*}
\sum_{q=1}^{+\infty }q^{2}\Vert J_{q}(F)\Vert _{L^{2}(\Omega )}^{2}<+\infty
\ .
\end{equation*}%
The pseudo-inverse operator of $L$ is defined as
\begin{equation*}
L^{-1}=-\sum_{q=1}^{\infty }\frac{1}{q}J_{q}
\end{equation*}%
and satisfies for each $F\in L^{2}(\Omega )$
\begin{equation*}
LL^{-1}F=F-{\mathbb{\ E}}[F]\text{ . }
\end{equation*}

\noindent The connection between stochastic calculus and probability metrics
is summarized in the following celebrated result (see e.g. \cite{noupebook},
Theorem 5.1.3), which will provide the basis for most of our results to
follow.

\begin{prop}
\label{BIGnourdinpeccati} 
Let $F\in
\mathbb{D}^{1,2}$ such that $\mathbb{E}[F]=0,$ $\mathbb{E}[F^{2}]=\sigma
^{2}<+\infty .$ Then we have%
\begin{equation*}
d_{W}(F,\mathcal{N}(0,1))\leq \sqrt{\frac{2}{\sigma ^{2}\,\pi }}\mathbb{E}%
[\left\vert \sigma ^{2}-\langle DF,-DL^{-1}F\rangle _{H}\right\vert ]\text{ .%
}
\end{equation*}%
Also, assuming in addition that $F$ has a density%
\begin{eqnarray*}
d_{TV}(F,\mathcal{N}(0,1)) &\leq &\frac{2}{\sigma ^{2}}\mathbb{E}[\left\vert
\sigma ^{2}-\langle DF,-DL^{-1}F\rangle _{H}\right\vert ]\text{ ,} \\
d_{K}(F,\mathcal{N}(0,1)) &\leq &\frac{1}{\sigma ^{2}}\mathbb{E}[\left\vert
\sigma ^{2}-\langle DF,-DL^{-1}F\rangle _{H}\right\vert ]\text{ .}
\end{eqnarray*}%
Moreover if $F\in \mathbb{D}^{1,4}$, we have also%
\begin{equation*}
\mathbb{E}[\left\vert \sigma ^{2}-\langle DF,-DL^{-1}F\rangle
_{H}\right\vert ]\leq \sqrt{\mathrm{Var}[\langle DF,-DL^{-1}F\rangle _{H}]}%
\text{ .}
\end{equation*}
\end{prop}
Furthermore, in the special case where $F=I_{q}(f)$ for $f\in H^{\odot q}$,
then from \cite{noupebook}, Theorem 5.2.6
\begin{equation}
\mathbb{E}[\left\vert \sigma ^{2}-\langle DF,-DL^{-1}F\rangle
_{H}\right\vert ]\leq \sqrt{\frac{1}{q^{2}}\sum_{r=1}^{q-1}r^{2}r!^{2}{%
\binom{q}{r}}^{4}(2q-2r)!\Vert f\widetilde{\otimes }_{r}f\Vert _{H^{\otimes
2q-2r}}^{2}}\ .  \label{casoparticolare}
\end{equation}
Note that in (\ref{casoparticolare}) we can replace $\Vert f\widetilde{%
\otimes }_{r}f\Vert _{H^{\otimes 2q-2r}}^{2}$ with the norm of the unsymmetryzed
contraction $\Vert f\otimes _{r}f\Vert _{H^{\otimes 2q-2r}}^{2}$ for the upper
bound, because $\Vert f\widetilde{\otimes }_{r}f\Vert _{H^{\otimes
2q-2r}}^{2}\leq \Vert f\otimes _{r}f\Vert _{H^{\otimes 2q-2r}}^{2}$ by the
triangular inequality.

\subsection{Polynomial transforms in Wiener chaoses}

As mentioned earlier in \S \ref{main results}, we shall be concerned first
with random variables $h_{\ell ;q,d}$, $\ell \geq 1$, $q,d\geq 2$
\begin{equation*}
h_{\ell ;q,d}=\int_{{\mathbb{S}}^{d}}H_{q}(T_{\ell }(x))\,dx\ ,
\end{equation*}%
and their (finite) linear combinations
\begin{equation}
Z_{\ell }=\sum_{q=2}^{Q}\beta _{q}h_{\ell ;q,d}\text{ ,}\qquad \beta _{q}\in
\mathbb{R},Q\in \mathbb{N}\ .  \label{genovese}
\end{equation}
Our first objective is to represent \paref{genovese} as a (finite) sum of
(multiple) stochastic integrals as in (\ref{chaos exp2}), in order to apply
the results recalled in \S \ref{background}.1. More explicitly, we shall
first provide the isonormal representation (\ref{isonormal}) on $L^{2}({%
\mathbb{S}}^{d})
$
for the Gaussian random eigenfunctions $T_{\ell }$, $\ell\ge 1$ i.e.,
we shall show that the following identity in law holds:%
\begin{equation*}
T_{\ell }(x)\overset{\mathcal L}{=}\int_{{\mathbb{S}}^{d}}\sqrt{\frac{n_{\ell ;d}}{%
\mu _{d}}}G_{\ell ;d}(\cos d(x,y))\,dW(y)\text{ ,}\qquad x\in \cS^d\ ,
\end{equation*}%
where $W$ is a Gaussian white noise on ${\mathbb{S}}^{d}$. To compare with (%
\ref{isonormal}), $T_{\ell }(x)=T(f_{x})$, where $T$ is the isonormal Gaussian field
on $L^2(\cS^d)$ and $f_{x}(\cdot ):=\sqrt{\frac{%
n_{\ell ;d}}{\mu _{d}}}G_{\ell ;d}(\cos d(x,\cdot ))$. Moreover we have
immediately that
\begin{equation*}
{\mathbb{E}}\left[\int_{{\mathbb{S}}^{d}}\sqrt{\frac{n_{\ell ;d}}{\mu _{d}}}%
G_{\ell ;d}(\cos d(x,y))\,dW(y)\right]=0\ ,
\end{equation*}%
and by the reproducing formula for Gegenbauer polynomials (\cite{szego})
$$\displaylines{
{\mathbb{E}}\left[\int_{{\mathbb{S}}^{d}}\sqrt{\frac{n_{\ell ;d}}{\mu _{d}}}%
G_{\ell ;d}(\cos d(x_{1},y_{1}))\,dW(y_{1})\int_{{\mathbb{S}}^{d}}\sqrt{%
\frac{n_{\ell ;d}}{\mu _{d}}}G_{\ell ;d}(\cos d(x_{2},y_{2}))\,dW(y_{2})%
\right] = \cr
=\frac{n_{\ell ;d}}{\mu _{d}}\int_{{\mathbb{S}}^{d}}G_{\ell ;d}(\cos
d(x_{1},y))G_{\ell ;d}(\cos d(x_{2},y))dy=G_{\ell ;d}(\cos d(x_{1},x_{2}))%
\text{ .}
}$$
Note that by (\ref{isometria}), we also have
\begin{equation*}
\displaylines{ H_{q}(T_{\ell }(x))=I_{q}(f_{x}^{\otimes q})=\cr
=\int_{({\mathbb{S}}^{d})^{q}}\left( \frac{n_{\ell;d}}{\mu _{d}}\right)
^{q/2}G_{\ell ;d}(\cos d(x,y_{1}))\dots G_{\ell ;d}(\cos
d(x,y_{q}))\,dW(y_{1})...dW(y_{q})\text{ ,} }
\end{equation*}%
so that
\begin{equation*}
h_{\ell ;q,d}\overset{\mathcal L}{=}\int_{({\mathbb{S}}^{d})^{q}}g_{\ell
;q}(y_{1},...,y_{q})\,dW(y_{1})...dW(y_{q})\ ,
\end{equation*}%
where%
\begin{equation}
g_{\ell ;q}(y_{1},...,y_{q}):=\int_{{\mathbb{S}}^{d}}\left( \frac{n_{\ell ;d}%
}{\mu _{d}}\right) ^{q/2}G_{\ell ;d}(\cos d(x,y_{1}))\dots G_{\ell ;d}(\cos
d(x,y_{q}))\,dx\text{ .}  \label{moltobello}
\end{equation}%
Thus we just established that $h_{\ell ;q,d}\overset{\mathcal L}{=}I_{q}(g_{\ell ;q})$
and therefore
\begin{equation}
Z_{\ell }\overset{\mathcal L}{=}\sum_{q=2}^{Q}I_{q}(\beta _{q}\,g_{\ell ;q})\ ,
\end{equation}%
as required. It should be noted that for such random variables $Z_{\ell }$,
the conditions of the Proposition \ref{BIGnourdinpeccati} are trivially
satisfied.

\section{On the variance of $h_{\ell ;q,d}$}

In this section we study the variance  of $h_{\ell ;q,d}$ defined in \eqref{hq}. By (\ref%
{hermite orto}) and the definition of Gaussian random eigenfunctions
\eqref{defT}, it follows that \eqref{int var} hold at once:
\begin{equation*}
\displaylines{ \Var[h_{\ell;q,d}]= \E \left[ \left( \int_{\cS^d}
H_q(T_\ell(x))\,dx \right)^2 \right] = \int_{(\cS^d)^2} \E[ H_q(T_\ell(x_1))
H_q(T_\ell(x_2))]\,dx_1 dx_2 = \cr = q! \int_{(\cS^d)^2} \E[T_\ell(x_1)
T_\ell(x_2)]^q\,dx_1 dx_2 = q! \int_{(\cS^d)^2} G_{\ell;d}(\cos
d(x_1,x_2))^q\,dx_1 dx_2=\cr = q! \mu_d \mu_{d-1} \int_0^{\pi}
G_{\ell;d}(\cos \vartheta)^q (\sin \vartheta)^{d-1}\, d\vartheta. }
\end{equation*}%
Now we prove Proposition \ref{varianza}, inspired by the proof of \cite%
{MaWi2}, Lemma $5.2$.

\subsection{Proof Proposition \protect\ref{varianza}}

\begin{proof}

By the Hilb's asymptotic formula for Jacobi polynomials (see \cite{szego}, Theorem $8.21.12$),
we have uniformly for $\ell\ge 1$, $\vartheta\in [0, \tfrac{\pi}2]$
$$\displaylines{
(\sin \vartheta)^{\frac{d}{2} - 1}G_{\ell ;d} (\cos \vartheta) =
\frac{2^{\frac{d}{2} - 1}}{{\ell + \frac{d}{2}-1\choose \ell}}
\left( a_{\ell, d} \left(\frac{\vartheta}{\sin
\vartheta}\right)^{\tfrac12}J_{\frac{d}{2} - 1}(L\vartheta) +
\delta(\vartheta) \right)\ , }$$ where $L=\ell + \frac{d-1}{2}$,
\begin{equation}\label{al}
a_{\ell, d} = \frac{\Gamma(\ell + \frac{d}{2})}{(\ell + \frac{d-1}{2})^{\tfrac{d}{2}-1} \ell !}\
\sim\  1\quad \text{as}\ \ell \to \infty,
\end{equation}
and the remainder is
\begin{equation*}
\delta(\vartheta) \ll \begin{cases} \sqrt{\vartheta}\,
\ell^{-\tfrac32}\ & \qquad \ell^{-1}< \vartheta
< \tfrac{\pi}2\ ,             \\
\vartheta^{\left(\tfrac{d}2-1\right) + 2}\, \ell^{\tfrac{d}2-1}\ & \qquad 0 < \vartheta
< \ell^{-1}\ .
\end{cases}
\end{equation*}
Therefore, in light of \eqref{al} and $\vartheta \to \frac{\vartheta}{\sin \vartheta}$ being bounded,
\begin{equation}\label{eq:G moment main error}
\begin{split}
\int_{0}^{\frac{\pi}{2}} G_{\ell ;d} (\cos \vartheta)^q &(\sin
\vartheta)^{d-1}d\vartheta
=\left(\frac{2^{\frac{d}{2} - 1}}{{\ell + \frac{d}{2}-1\choose
\ell}}\right)^q a^q_{\ell,d} \int_0^{\frac{\pi}{2}} ( \sin
\vartheta)^{- q(\frac{d}{2} -1)}
 \Big( \frac{\vartheta}{\sin \vartheta} \Big)^{\frac{q}{2}}
J^q_{\frac{d}{2}-1}(L\vartheta) (\sin \vartheta)^{d-1} d\vartheta\ +\cr
&+O\left(\frac{1}{\ell^{q(\frac{d}{2}-1)}} \int_0^{\frac{\pi}{2}}
( \sin \vartheta)^{- q(\frac{d}{2} -1)}
|J_{\frac{d}{2}-1}(L\vartheta)|^{q-1}
\delta(\vartheta)(\sin \vartheta)^{d-1}d\vartheta\right),
\end{split}
\end{equation}
where we used $${\ell + \frac{d}{2}-1\choose \ell} \ll \frac{1}{\ell^{\frac{d}{2}-1}}$$
(note that we readily neglected the smaller terms, corresponding to higher powers of $\delta(\vartheta)$).
We rewrite \eqref{eq:G moment main error} as
\begin{equation}
\label{eq:G moment=M+E}
\begin{split}
\int_{0}^{\frac{\pi}{2}} G_{\ell ;d} (\cos \vartheta)^q (\sin
\vartheta)^{d-1} d\vartheta
=N+E\ ,
\end{split}
\end{equation}
where
\begin{equation}
\label{eq:Mdql def}
N=N(d,q;\ell) := \left(\frac{2^{\frac{d}{2} - 1}}{{\ell + \frac{d}{2}-1\choose
\ell}}\right)^q a^q_{\ell,d} \int_0^{\frac{\pi}{2}} ( \sin
\vartheta)^{- q(\frac{d}{2} -1)}
 \Big( \frac{\vartheta}{\sin \vartheta} \Big)^{\frac{q}{2}}
J_{\frac{d}{2}-1}(L\vartheta)^q (\sin \vartheta)^{d-1} d\vartheta\
\end{equation}
and
\begin{equation}
\label{eq:Edql def}
E=E(d,q;\ell) \ll \frac{1}{\ell^{q(\frac{d}{2}-1)}} \int_0^{\frac{\pi}{2}}
( \sin \vartheta)^{- q(\frac{d}{2} -1)}
|J_{\frac{d}{2}-1}(L\vartheta)|^{q-1}
\delta(\vartheta)(\sin \vartheta)^{d-1}d\vartheta\ .
\end{equation}
To bound the error term $E$ we split the range of the integration in \eqref{eq:Edql def}
and write
\begin{equation}
\label{eq:int error split}
\begin{split}
E \ll & \frac{1}{\ell^{q(\frac{d}{2}-1)}}
\int\limits_{0}^{\frac{1}{\ell}}( \sin \vartheta)^{- q(\frac{d}{2} -1)}
|J_{\frac{d}{2}-1}(L\vartheta)|^{q-1}
\vartheta^{\left(\tfrac{d}2-1\right) + 2}\, \ell^{\tfrac{d}2-1}(\sin \vartheta)^{d-1}\,d\vartheta +
\\&+\frac{1}{\ell^{q(\frac{d}{2}-1)}}
\int_{\frac{1}{\ell}}^{\frac{\pi}{2}}
( \sin \vartheta)^{- q(\frac{d}{2} -1)}
|J_{\frac{d}{2}-1}(L\vartheta)|^{q-1}
\sqrt{\vartheta}\,
\ell^{-\tfrac32}(\sin \vartheta)^{d-1}\,d\vartheta\ .
\end{split}
\end{equation}
For the first integral in \eqref{eq:int error split}
recall that $J_{\frac{d}{2}-1}(z) \sim z^{\frac{d}{2}-1}$ as
$z\to 0$, so that as $\ell\to \infty$,
$$\displaylines{
\frac{1 }{\ell^{(q-1)(\frac{d}{2}-1)}} \int_0^{\frac{1}{\ell}}
\left(\frac{\vartheta}{ \sin \vartheta}\right)^{ q(\frac{d}{2} -1)-d+1}
|J_{\frac{d}{2}-1}(L\vartheta)|^{q-1}
\vartheta^{-(q-1)\left(\tfrac{d}2-1\right) + d+1}\, \,d\vartheta
\ll\cr
}$$
\begin{equation}\label{err1}
\ll \int_0^{\frac{1}{\ell}}\vartheta^{d+1}\,d\vartheta = \frac{1}{\ell^{d+2}}\ ,
\end{equation}
which is enough for our purposes. Furthermore,
since for $z$ big $|J_{\frac{d}{2}-1}(z)|=O(z^{-\tfrac12})$ (and keeping in mind
that $L$ is of the same order of magnitude as $\ell$), we may bound the second integral
in \eqref{eq:int error split} as
$$\displaylines{\ll
\frac{1}{\ell^{q(\frac{d}{2}-1)+\frac32}}\int_{\frac{1}{\ell}}^{\frac{\pi}{2}}
\left( \frac{\vartheta}{\sin \vartheta}\right)^{ q(\frac{d}{2} -1)-d+1}
|J_{\frac{d}{2}-1}(L\vartheta)|^{q-1}
\vartheta^{-q(\frac{d}{2}-1)+d-\frac12}\,d\vartheta \ll \cr
\ll
\frac{1}{\ell^{q(\frac{d}{2}-1)+\frac32}}
\int_{\frac{1}{\ell}}^{\frac{\pi}{2}}
(\ell\vartheta)^{-\frac{q-1}{2}}
\vartheta^{-q(\frac{d}{2}-1)+d-\frac12}\,d\vartheta
=\frac{1}{\ell^{q(\frac{d}{2}-\frac12)+2}}
\int_{\frac{1}{\ell}}^{\frac{\pi}{2}}
\vartheta^{-q(\frac{d}{2}-\frac12)+d}\,d\vartheta \ll \cr
}$$
\begin{equation}\label{err2}
\ll \frac{1}{\ell^{(d+2)\wedge \left(q\left(\tfrac{d}2 -\tfrac{1}2\right) +1\right)}} = o(\ell^{-d})\ ,
\end{equation}
where the last equality in \paref{err2} holds for $q\ge 3$.
From \paref{err1} (bounding the first integral in \eqref{eq:int error split}) and
\paref{err2} (bounding the second integral in \eqref{eq:int error split}) we finally find that the error term in
\eqref{eq:G moment=M+E} is
\begin{equation}
\label{resto}
E =o(\ell^{-d})
\end{equation}
for $q\ge 3$, admissible for our purposes.

Therefore, substituting \paref{resto} into \eqref{eq:G moment=M+E} we have
$$\displaylines{
\int_{0}^{\frac{\pi}{2}} G_{\ell ;d} (\cos \vartheta)^q (\sin
\vartheta)^{d-1}\, d\vartheta =\cr
=\left(\frac{2^{\frac{d}{2} -
1}}{{\ell + \frac{d}{2}-1\choose \ell}}\right)^q a^q_{\ell,d}
\int_0^{\frac{\pi}{2}} ( \sin \vartheta)^{- q(\frac{d}{2} -1)}
 \Big( \frac{\vartheta}{\sin \vartheta} \Big)^{\frac{q}{2}}
J_{\frac{d}{2}-1}(L\vartheta)^q (\sin \vartheta)^{d-1}d\vartheta + o(\ell^{-d}) =
}$$
\begin{equation}
\label{bene}
=\left(\frac{2^{\frac{d}{2} -
1}}{{\ell + \frac{d}{2}-1\choose \ell}}\right)^q a^q_{\ell,d} \frac{1}{L}
\int_0^{L\frac{\pi}{2}} ( \sin \psi/L)^{- q(\frac{d}{2} -1)}
 \Big( \frac{\psi/L}{\sin \psi/L} \Big)^{\frac{q}{2}}
J_{\frac{d}{2}-1}(\psi)^q (\sin \psi/L)^{d-1}\, d\psi + o(\ell^{-d})\ ,
\end{equation}
where in the last equality we transformed $\psi/L=\vartheta$; it then remains to evaluate
the first term in \paref{bene}, which we denote by
\begin{equation*}
N_L := \left(\frac{2^{\frac{d}{2} -
1}}{{\ell + \frac{d}{2}-1\choose \ell}}\right)^q a^q_{\ell,d} \frac{1}{L}
\int_0^{L\frac{\pi}{2}} ( \sin \psi/L)^{- q(\frac{d}{2} -1)}
 \Big( \frac{\psi/L}{\sin \psi/L} \Big)^{\frac{q}{2}}
J_{\frac{d}{2}-1}(\psi)^q (\sin \psi/L)^{d-1}\, d\psi\ .
\end{equation*}
Now recall that as $\ell\to \infty$
\begin{equation*}
{\ell + \frac{d}{2}-1\choose \ell}\ \sim \ \frac{\ell^{\frac{d}{2}-1}}{(\frac{d}2-1)!}\ ;
\end{equation*}
moreover \paref{al} holds,   therefore we find of course that as $L\to \infty$
\begin{equation}\label{mado}
N_L\  \sim\  \frac{(2^{\frac{d}{2} -
1}(\frac{d}2-1)!)^q}{L^{q(\frac{d}{2}-1)+1}}
\int_0^{L\frac{\pi}{2}} ( \sin \psi/L)^{- q(\frac{d}{2} -1)}
 \Big( \frac{\psi/L}{\sin \psi/L} \Big)^{\frac{q}{2}}
J_{\frac{d}{2}-1}(\psi)^q (\sin \psi/L)^{d-1}\, d\psi\ .
\end{equation}
In order to finish the proof of Proposition \ref{varianza},
it is enough to check that, as $L\to \infty$
$$\displaylines{
L^d \, \frac{(2^{\frac{d}{2} -
1}(\frac{d}2-1)!)^q}{L^{q(\frac{d}{2}-1)+1}}
\int_0^{L\frac{\pi}{2}} ( \sin \psi/L)^{- q(\frac{d}{2} -1)}
 \Big( \frac{\psi/L}{\sin \psi/L} \Big)^{\frac{q}{2}}
J_{\frac{d}{2}-1}(\psi)^q (\sin \psi/L)^{d-1}\, d\psi\, \goto\, c_{q;d}\ ,
}$$
actually from \paref{bene} and \paref{mado}, we have
$$\displaylines{
\lim_{\ell\to +\infty} \ell^d \int_{0}^{\frac{\pi}{2}} G_{\ell ;d} (\cos \vartheta)^q (\sin
\vartheta)^{d-1}\, d\vartheta = \cr
= \lim_{L\to +\infty} L^d \, \frac{(2^{\frac{d}{2} -
1}(\frac{d}2-1)!)^q}{L^{q(\frac{d}{2}-1)+1}}
\int_0^{L\frac{\pi}{2}} ( \sin \psi/L)^{- q(\frac{d}{2} -1)}
 \Big( \frac{\psi/L}{\sin \psi/L} \Big)^{\frac{q}{2}}
J_{\frac{d}{2}-1}(\psi)^q (\sin \psi/L)^{d-1}\, d\psi \ .
}$$
Now we write
$$\frac{\psi/L}{\sin \psi/L} = 1 + O\left( \psi^2/L^2  \right),$$
so that
$$\displaylines{
L^d\, \frac{\left(2^{\frac{d}{2} - 1}(\frac{d}2-1)!\right)^q}{L^{q(\frac{d}{2}-1)+1}}
\int_0^{L\frac{\pi}{2}} ( \sin \psi/L)^{- q(\frac{d}{2} -1)}
 \Big( \frac{\psi/L}{\sin \psi/L} \Big)^{\frac{q}{2}}
J_{\frac{d}{2}-1}(\psi)^q (\sin \psi/L)^{d-1}\, d\psi=\cr
=\left(2^{\frac{d}{2} - 1}(\frac{d}2-1)!\right)^q
\int_0^{L\frac{\pi}{2}}
 \Big( \frac{\psi/L}{\sin \psi/L} \Big)^{q(\tfrac{d}{2} -\tfrac12)-d+1}
J_{\frac{d}{2}-1}(\psi)^q \psi^{- q(\frac{d}{2} -1)+d-1}\, d\psi=\cr
=\left(2^{\frac{d}{2} - 1}(\frac{d}2-1)!\right)^q
\int_0^{L\frac{\pi}{2}}
 \Big( 1 + O\left( \psi^2/L^2  \right) \Big)^{q(\tfrac{d}{2} -\tfrac12)-d+1}
J_{\frac{d}{2}-1}(\psi)^q \psi^{- q(\frac{d}{2} -1)+d-1}\, d\psi=\cr
=\left(2^{\frac{d}{2} - 1}(\frac{d}2-1)!\right)^q
\int_0^{L\frac{\pi}{2}}
J_{\frac{d}{2}-1}(\psi)^q \psi^{- q(\frac{d}{2} -1)+d-1}\, d\psi +\cr
+
O\left( \frac{1}{L^2}\int_0^{L\frac{\pi}{2}}
J_{\frac{d}{2}-1}(\psi)^q \psi^{- q(\frac{d}{2} -1)+d+1}\,d\psi  \right).
}$$
Note that as  $L\to +\infty$, the first term of the previous summation
 converges to $c_{q;d}$ defined in \paref{cq}, i.e.
\begin{equation}\label{to cq}
\left(2^{\frac{d}{2} - 1}(\frac{d}2-1)!\right)^q
\int_0^{L\frac{\pi}{2}}
J_{\frac{d}{2}-1}(\psi)^q \psi^{- q(\frac{d}{2} -1)+d-1}\, d\psi \to c_{q;d}\ .
\end{equation}
It remains to bound the remainder
$$
 \frac{1}{L^2}\int_0^{L\frac{\pi}{2}}
|J_{\frac{d}{2}-1}(\psi)|^q \psi^{- q(\frac{d}{2} -1)+d+1}\,d\psi =
O(1) +  \frac{1}{L^2}\int_1^{L\frac{\pi}{2}}
|J_{\frac{d}{2}-1}(\psi)|^q \psi^{- q(\frac{d}{2} -1)+d+1}\,d\psi\ .
$$
Now for the second term on the r.h.s.
$$\displaylines{
\int_1^{L\frac{\pi}{2}}
|J^q_{\frac{d}{2}-1}(\psi)| \psi^{- q(\frac{d}{2} -1)+d+1}\,d\psi \ll
\int_1^{L\frac{\pi}{2}} \psi^{- q(\frac{d}{2} -\frac{1}{2})+d+1}\,d\psi=\cr
= O(1 + L^{- q(\frac{d}{2} -\frac{1}{2})+d+2})\ .
}$$
Therefore we obtain
$$\displaylines{
\left(2^{\frac{d}{2} - 1}(\frac{d}2-1)!\right)^q \int_0^{L\frac{\pi}{2}}
J_{\frac{d}{2}-1}(\psi)^q \psi^{- q(\frac{d}{2} -1)+d-1}\, d\psi +
O\left( \frac{1}{L^2}\int_0^{L\frac{\pi}{2}}
J_{\frac{d}{2}-1}(\psi)^q \psi^{- q(\frac{d}{2} -1)+d+1}\,d\psi
\right)=\cr
= \left(2^{\frac{d}{2} - 1}(\frac{d}2-1)!\right)^q
\int_0^{L\frac{\pi}{2}} J_{\frac{d}{2}-1}(\psi)^q \psi^{-
q(\frac{d}{2} -1)+d-1}\, d\psi + O(L^{-2} + L^{- q(\frac{d}{2}
-\frac{1}{2})+d})\ , }$$ so that we have just checked the
statement of the present proposition for $q > \frac{2d}{d-1}$. This is indeed enough for each $q\ge
3$  when
$d\ge 4$ .

It remains to investigate separately just the case $d=q=3$.
Recall that for $d=3$ we have an explicit formula for the Bessel
function of order $\frac{d}{2}-1$ (\cite{szego}), that is
\begin{equation*}
J_{\frac{1}{2}}(z) = \sqrt{\frac{2}{\pi z}} \sin (z)\ ,
\end{equation*}
and hence the integral in \paref{cq} is indeed convergent for $q=d=3$ by integrations by parts.

We have hence to study the convergence of the following integral
\begin{equation*}
\frac{8}{\pi^\frac{3}{2} }
\int_0^{L\frac{\pi}{2}}
\left( \frac{\psi/L}{\sin \psi/L} \right)
\frac{\sin^3 \psi}{\psi} \, d\psi\ .
\end{equation*}
To this aim, let us consider a large parameter $K\gg 1$ and divide the integration range
into $[0,K]$ and $[K, \frac{\pi}{2}]$;
the main contribution comes from the first term, whence we have to prove that the latter vanishes.
Note that
\begin{equation}
\label{eq:int K-> bound}
\int_K^{L\frac{\pi}{2}} \left( \frac{\psi/L}{\sin \psi/L} \right)
\frac{\sin^3 \psi}{\psi} \, d\psi \ll \frac{1}{K}\ ,
\end{equation}
where we use integration by part with the bounded function $I(T) = \int_0^T \sin^3 z\, dz$.
On $[0,K]$, we write
$$\displaylines{
\frac{8}{\pi^\frac{3}{2} }\int_0^{K} \left( \frac{\psi/L}{\sin \psi/L} \right)
\frac{\sin^3 \psi}{\psi} \, d\psi
= \frac{8}{\pi^\frac{3}{2} }\int_0^{K}
\frac{\sin^3 \psi}{\psi} \, d\psi + O\left( \frac{1}{L^2}\int_0^{K}
\psi \sin^3 \psi\, d\psi   \right) = \cr
=\frac{8}{\pi^\frac{3}{2} }\int_0^{K}
\frac{\sin^3 \psi}{\psi} \, d\psi +O\left( \frac{K^2}{L^2} \right).
}$$
Consolidating the latter with \eqref{eq:int K-> bound} we find that
$$\displaylines{
\frac{8}{\pi^\frac{3}{2} }
\int_0^{L\frac{\pi}{2}}
\left( \frac{\psi/L}{\sin \psi/L} \right)
\frac{\sin^3 \psi}{\psi} \, d\psi = \frac{8}{\pi^\frac{3}{2} }\int_0^{K}
\frac{\sin^3 \psi}{\psi} \, d\psi + O\left(  \frac{1}{K} +  \frac{K^2}{L^2} \right).
}$$
Now as $K\to +\infty$,
$$
\frac{8}{\pi^\frac{3}{2} }\int_0^{K}
\frac{\sin^3 \psi}{\psi} \, d\psi \to c_{3;3}\ ;
$$
to conclude the proof, it is then enough to choose $K=K(L)\rightarrow\infty$
sufficiently slowly, i.e. $K=\sqrt{L}$.
\end{proof}

\section{The quantitative Central Limit Theorem for $h_{\ell;q,d}$}

In this section we prove Theorem \ref{teo1} with the help of Proposition \ref%
{BIGnourdinpeccati} and \eqref{casoparticolare} in particular. The
identifications of \S 2.2 lead to some very explicit expressions for the
contractions \eqref{contrazione}, as in the following result.

For $\ell\ge 1, q\ge 2$, let $g_{\ell ;q}$ be defined as in \eqref{moltobello}.

\begin{lemma}
\label{contractions}
For all $q_{1},q_{2}\ge 2$, $r=1,...,q_{1}\wedge q_{2}-1$,
we have the identities
$$\displaylines{ \left\Vert g_{\ell ;q_{1}}\otimes _{r}g_{\ell
;q_{2}}\right\Vert _{H^{\otimes
n}}^{2}=\cr
= \int_{(\cS^{d})^4}G_{\ell
;d}^{r}(\cos d( x_{1},x_{2}) ) G_{\ell
;d}^{q_{1}\wedge q_{2}-r}(\cos d( x_{2},x_{3}) )   G_{\ell
;d}^{r}(\cos d(x_{3},x_{4}) )      G_{\ell
;d}^{q_{1}\wedge q_{2}-r}(\cos d( x_{1},x_{4}) )\,d\underline{x}\ , }
$$
where we set $d\underline{x} := dx_{1}dx_{2}dx_{3}dx_{4}$ and
$n:= q_{1}+q_{2}-2r$.
\end{lemma}

\begin{proof}
Assume w.l.o.g. $q_1\le q_2$ and set for simplicity of notation $d\underline{t}:=dt_{1}\dots dt_{r}$.
The contraction \paref{contrazione} here takes the form%
$$\displaylines{
(g_{\ell;q_{1} }\otimes _{r}g_{\ell;q_{2}})(y_{1},...,y_{n})=\cr
=\int_{(\cS^d)^r} g_{\ell;q_{1} }(y_{1},\dots,y_{q_{1}-r},t_{1},\dots,t_{r})g_{\ell;q_{2}
}(y_{q_{1}-r+1},\dots,y_{n},t_{1},\dots,t_{r})\,d\underline{t}=\cr
=\int_{(\cS^d)^r}  \int_{\cS^{d}}\left(\frac{n_{\ell; d}}{\mu_d }
\right) ^{q_1/2} G_{\ell ;d}(\cos d(
x_{1},y_{1}) )\dots 
G_{\ell
;d}(\cos d( x_{1},t_{r}) )\,dx_{1}\times \cr
\times \int_{\cS^{d}}\left(\frac{n_{\ell; d}}{\mu_d }
\right) ^{q_2/2}G_{\ell ;d}(\cos d(
x_{2},y_{q_{1}-r+1}) )\dots 
G_{\ell ;d}(\cos d(
x_{2},t_{r}) )\,dx_{2}\, d\underline{t}=\cr
=\int_{(\cS^{d})^2}\left(\frac{n_{\ell; d}}{\mu_d }
\right) ^{n/2}G_{\ell ;d}(\cos d(
x_{1},y_{1}) )\dots  G_{\ell ;d}(\cos d(
x_{1},y_{q_1-r}) )\times \cr
\times G_{\ell ;d}(\cos d(
x_{2},y_{q_{1}-r+1}) )\dots G_{\ell ;d}(\cos d(
x_{2},y_{n}) )G_{\ell
;d}^{r}(\cos d( x_{1},x_{2}) )\,dx_{1}dx_{2}\ ,
}$$
where in the last equality
we have repeatedly used the reproducing property of Gegenbauer polynomials (\cite{szego}).
Now set $d\underline{y}:=dy_{1}\dots dy_{n}$. It follows at once that
$$\displaylines{
\left\Vert g_{\ell; q_{1} }\otimes _{r}g_{\ell; q_{2} }\right\Vert _{%
{H}^{\otimes n}}^{2}=\cr
=\int_{(\cS^d)^{n}} (g_{\ell; q_{1} }\otimes _{r}g_{\ell; q_{2} })^{2}
(y_{1},\dots,y_{n})\,d\underline{y}=\cr
=\int_{(\cS^d)^{n}} \int_{(\cS^{d})^2}\left(\frac{n_{\ell; d}}{\mu_d
} \right) ^{n}G_{\ell ;d}(\cos d(
x_{1},y_{1}) )\dots G_{\ell ;d}(\cos d(
x_{2},y_{n}) )G_{\ell
;d}^{r}(\cos d(x_{1},x_{2}) )dx_{1}dx_{2}\times \cr
\times
 \int_{(\cS^{d})^2}G_{\ell ;d}(\cos d(
x_{4},y_{1}) )\dots G_{\ell ;d}(\cos d(
x_{3},y_{n}) )G_{\ell
;d}^{r}(\cos d( x_{3},x_{4})
)\,dx_{3}dx_{4}\,d\underline{y}=\cr
=\int_{(\cS^{d})^4}G_{\ell
;d}^{r}(\cos d( x_{1},x_{2}) ) G_{\ell
;d}^{q_{1}-r}(\cos d( x_{2},x_{3}) )   G_{\ell
;d}^{r}(\cos d(x_{3},x_{4}) )      G_{\ell
;d}^{q_{1}-r}(\cos d( x_{1},x_{4}) )
\,d\underline{x}\ ,
}$$
as claimed.
\end{proof}
We need now to introduce some further notation, i.e. for $q\ge 2$ and $r= 1,\dots, q-1$
\begin{equation*}
\displaylines{ \mathcal{K}_{\ell }(q;r):=\int_{(\cS^{d})^4}G_{\ell
;d}^{r}(\cos d( x_{1},x_{2}) ) G_{\ell ;d}^{q-r}(\cos d( x_{2},x_{3})
)\times \cr \times G_{\ell ;d}^{r}(\cos d(x_{3},x_{4}) ) G_{\ell
;d}^{q-r}(\cos d( x_{1},x_{4}) ) \,dx_{1}dx_{2}dx_{3}dx_{4}, }
\end{equation*}%
Lemma \ref{contractions} asserts that
\begin{equation}\label{K}
\mathcal{K}_{\ell }(q;r)=\left\Vert g_{\ell ;q}\otimes _{r}g_{\ell
;q}\right\Vert _{H^{\otimes 2q-2r}}^{2}\text{ ;}
\end{equation}%
it is immediate to check that
\begin{equation}\label{simm}
\mathcal{K}_{\ell }(q;r)= \mathcal{K}_{\ell }(q;q-r)\ .
\end{equation}
In the following two propositions we bound each term of the form $\mathcal{K}(q;r)$  (from (\ref{simm})
it is enough to consider $r=1,\dots, \left[\frac{q}2\right]$).
As noted in
\S $1.1$, these bounds improve the existing literature even for the case $d=2,$
from which we start our analysis.

For $d=2,$ as previously recalled, Gegenbauer polynomials become standard
Legendre polynomials $P_{\ell },$ for which it is well-known that (see (\ref%
{momento 2}))
\begin{equation}
\int_{{\mathbb{S}}^{2}}P_{\ell }(\cos d( x_{1},x_{2}) )^{2}\,dx_{1}=O\left(
\frac{1}{\ell }\right) \text{ ;}  \label{tri1}
\end{equation}%
also, from \cite{MaWi12}, Lemma 3.2 we have that%
\begin{equation}
\int_{{\mathbb{S}}^{2}}P_{\ell }(\cos d( x_{1},x_{2}) )^{4}\,dx_{1}=O\left(
\frac{\log \ell }{\ell ^{2}}\right) \text{ .}  \label{tri2}
\end{equation}%
Finally, it is trivial to show that%
\begin{equation}
\int_{{\mathbb{S}}^{2}}\left\vert P_{\ell }(\cos d( x_{1},x_{2})
)\right\vert\, dx_{1}\leq \sqrt{\int_{{\mathbb{S}}^{2}}P_{\ell
}(\cos d( x_{1},x_{2}) )^{2}\,dx_{1}}=O\left( \frac{1}{\sqrt{%
\ell }}\right)   \label{tri3}
\end{equation}%
and%
\begin{equation}
\int_{{\mathbb{S}}^{2}}\left\vert P_{\ell }(\cos d( x_{2},x_{3})
)\right\vert ^{3}\,dx_{2}\leq \sqrt{\int_{{\mathbb{S}}^{2}}P_{\ell
}(\cos d( x_{2},x_{3}) )^{2}\,dx_{2}}\sqrt{\int_{{\mathbb{S}}^{2}}P_{\ell
}(\cos d( x_{1},x_{2}) )^{4}\,dx_{1}}=O\left( \sqrt{\frac{\log \ell }{\ell
^{3}}}\right) \text{ .}  \label{tri4}
\end{equation}%

\begin{prop}
\label{cum2} For all $r=1,2,\dots ,q-1,$ we have%
\begin{eqnarray}
 \mathcal{K}_{\ell }(q;r) &=& O\left(\frac{1}{\ell ^{5}}%
\right)\text{ for }q=3\text{ ,}  \label{hotel1} \\
\mathcal{K}_{\ell }(q;r) &=& O\left(\frac{1}{\ell ^{4}}%
\right)\text{ for }q=4\text{ ,}  \label{hotel2}\\
\mathcal{K}_{\ell }(q;r) &=& O\left(\frac{\log \ell }{%
\ell ^{9/2}}\right)\text{  for }q=5,6\text{ }  \label{hotel3}
\end{eqnarray}%
and%
\begin{equation}
\mathcal{K}_{\ell }(q;1) =\mathcal{K}_{\ell
}(q;q-1) =O\left(\frac{1}{\ell ^{9/2}}\right)\text{ , }
\mathcal{K}_{\ell }(q;r) =O\left(\frac{1}{\ell ^{5}}\right)\text{
, }r=2,...,q-2,\text{ for }q\geq 7\text{ .}  \label{hotel4}
\end{equation}
\end{prop}

\begin{proof}
The bounds (\ref{hotel1}), (\ref{hotel2}) are known and indeed the
corresponding integrals can be evaluated explicitly in terms of Wigner's 3j
and 6j coefficients, see \cite{M2008}, \cite{MaPeCUP}, \cite{MaWi12}. The
bounds in (\ref{hotel3}),(\ref{hotel4}) derives from a simple improvement in
the proof of Proposition 2.2 in \cite{MaWi12}, which can be obtained when
focussing only on a subset of the terms (the circulant ones) considered in
that reference. In the proof to follow, we exploit repeatedly
 (\ref{tri1}), (\ref{tri2}), (\ref{tri3}) and \paref{tri4}.

Let us start investigating the case $q=5$:%
$$\displaylines{
\mathcal{K}_{\ell }(5;1)=\int_{(\cS^{2})^4}\left\vert
P_{\ell }(\cos d( x_{1},x_{2}) )\right\vert ^{4}\left\vert
P_{\ell }(\cos d( x_{2},x_{3}) )\right\vert \times \cr
\times \left\vert
P_{\ell }(\cos d( x_{3},x_{4}) )\right\vert ^{4}\left\vert
P_{\ell }(\cos d( x_{4},x_{1}) )\right\vert
dx_{1}dx_{2}dx_{3}dx_{4} \le \cr
\mathcal{\leq }\int_{(\cS^{2})^4}\left\vert P_{\ell
}(\cos d( x_{1},x_{2}) )\right\vert ^{4}\left\vert P_{\ell
}(\cos d( x_{2},x_{3}) )\right\vert \left\vert P_{\ell
}(\cos d( x_{3},x_{4}) )\right\vert
^{4}dx_{1}dx_{2}dx_{3}dx_{4} \le \cr
\mathcal{\leq }\int_{(\cS^{2})^3}\left\vert P_{\ell
}(\cos d( x_{1},x_{2}) )\right\vert ^{4}\left\vert P_{\ell
}(\cos d( x_{2},x_{3}) )\right\vert \left\{
\int_{\cS^{2}}\left\vert P_{\ell }(\cos d( x_{3},x_{4})
)\right\vert ^{4}dx_{4}\right\} dx_{1}dx_{2}dx_{3} \le\cr
\mathcal{\leq } O\left( \frac{\log \ell }{\ell ^{2}}\right) \times
\int_{\cS^{2}\times \cS^{2}}\left\vert P_{\ell }(\cos d(
x_{1},x_{2}) )\right\vert ^{4}\left\{ \int_{\cS^{2}}\left\vert
P_{\ell }(\cos d( x_{2},x_{3}) )\right\vert dx_{3}\right\}
dx_{1}dx_{2} \le \cr
\leq O\left( \frac{\log \ell }{\ell ^{2}}\right) \times O\left( \frac{1}{%
\sqrt{\ell }}\right) \times \int_{\cS^{2}\times \cS^{2}}\left\vert P_{\ell
}(\cos d( x_{1},x_{2}) )\right\vert ^{4}dx_{1}dx_{2} \le \cr
\leq O\left( \frac{\log \ell }{\ell ^{2}}\right) \times O\left( \frac{1}{%
\sqrt{\ell }}\right) \times O\left( \frac{\log \ell }{\ell ^{2}}\right)
=O\left( \frac{\log ^{2}\ell }{\ell ^{9/2}}\right) \text{ ;}
}$$
$$\displaylines{
\mathcal{K}_{\ell }(5;2)=\int_{(\cS^{2})^4}\left\vert
P_{\ell }(\cos d( x_{1},x_{2}) )\right\vert ^{3}\left\vert
P_{\ell }(\cos d( x_{2},x_{3}) )\right\vert^2 \times \cr
\times
 \left\vert
P_{\ell }(\cos d( x_{3},x_{4}) )\right\vert ^{3}\left\vert
P_{\ell }(\cos d( x_{4},x_{1}) )\right\vert^2
dx_{1}dx_{2}dx_{3}dx_{4}\le \cr
\mathcal{\leq }\int_{(\cS^{2})^4}\left\vert P_{\ell
}(\cos d( x_{1},x_{2}) )\right\vert ^{3}\left\vert P_{\ell
}(\cos d( x_{2},x_{3}) )\right\vert^2 \left\vert P_{\ell
}(\cos d( x_{3},x_{4}) )\right\vert
^{3}dx_{1}dx_{2}dx_{3}dx_{4} \le \cr
\mathcal{\leq }\int_{(\cS^{2})^3}\left\vert P_{\ell
}(\cos d( x_{1},x_{2}) )\right\vert ^{3}\left\vert P_{\ell
}(\cos d( x_{2},x_{3}) )\right\vert^2 \left\{
\int_{\cS^{2}}\left\vert P_{\ell }(\cos d( x_{3},x_{4})
)\right\vert ^{3}dx_{4}\right\} dx_{1}dx_{2}dx_{3} \le \cr
\mathcal{\leq }O\left( \sqrt{\frac{\log \ell }{\ell ^{3}}}\right) \times
\int_{\cS^{2}\times \cS^{2}}\left\vert P_{\ell }(\cos d(
x_{1},x_{2}) )\right\vert ^{3}\left\{ \int_{\cS^{2}}\left\vert
P_{\ell }(\cos d( x_{2},x_{3}) )\right\vert^2 dx_{3}\right\}
dx_{1}dx_{2} \le \cr
\leq O\left( \sqrt{\frac{\log \ell }{\ell ^{3}}}\right) \times O\left( \frac{1}{%
\ell }\right) \times \int_{\cS^{2}\times \cS^{2}}\left\vert P_{\ell
}(\cos d( x_{1},x_{2}) )\right\vert ^{3}dx_{1}dx_{2} \le \cr
\leq O\left( \sqrt{\frac{\log \ell }{\ell ^{3}}}\right)\times O\left( \frac{1}{%
\ell}\right) \times O\left( \sqrt{\frac{\log \ell }{\ell ^{3}}}\right)
=O\left( \frac{\log \ell }{\ell ^{4}}\right) \text{ .}
}$$
For $q=6$ and $r=1$ we simply note that $\mathcal{K}_{\ell }(6;1)\le \mathcal{K}_{\ell }(5;1)$, actually
$$\displaylines{
\mathcal{K}_{\ell }(6;1)=\int_{(\cS^{2})^4}\left\vert
P_{\ell }(\cos d( x_{1},x_{2}) )\right\vert ^{5}\left\vert
P_{\ell }(\cos d( x_{2},x_{3}) )\right\vert \times \cr
\times\left\vert
P_{\ell }(\cos d( x_{3},x_{4}) )\right\vert ^{5}\left\vert
P_{\ell }(\cos d( x_{4},x_{1}) )\right\vert
dx_{1}dx_{2}dx_{3}dx_{4}\le \cr
\leq \int_{(\cS^{2})^4}\left\vert P_{\ell }(\cos d(
x_{1},x_{2}) )\right\vert ^{4}\left\vert P_{\ell }(\cos d(
x_{2},x_{3}) )\right\vert \times \cr
\times \left\vert P_{\ell }(\cos d(
x_{3},x_{4}) )\right\vert ^{4}\left\vert P_{\ell }(\cos d(
x_{4},x_{1}) )\right\vert dx_{1}dx_{2}dx_{3}dx_{4}=
\mathcal{K}_{\ell }(5;1)=O\left( \frac{\log ^{2}\ell }{\ell ^{9/2}}\right)\ .
}$$
Then we find with analogous computations as for $q=5$ that
$$\displaylines{
\mathcal{K}_{\ell }(6;2)=\int_{(\cS^{2})^4}\left\vert
P_{\ell }(\cos d( x_{1},x_{2}) )\right\vert ^{4}\left\vert
P_{\ell }(\cos d( x_{2},x_{3}) )\right\vert ^{2}\times \cr
\times \left\vert
P_{\ell }(\cos d( x_{3},x_{4}) )\right\vert ^{4}\left\vert
P_{\ell }(\cos d( x_{4},x_{1}) )\right\vert
^{2}dx_{1}dx_{2}dx_{3}dx_{4}\le \cr
\le
\int_{(\cS^{2})^4}\left\vert P_{\ell
}(\cos d( x_{1},x_{2}) )\right\vert ^{4}\left\vert P_{\ell
}(\cos d( x_{2},x_{3}) )\right\vert ^{2}\times \cr
\times \left\vert P_{\ell
}(\cos d( x_{3},x_{4}) )\right\vert ^{4}\left\vert P_{\ell
}(\cos d( x_{4},x_{1}) )\right\vert
^{2}dx_{1}dx_{2}dx_{3}dx_{4}\le \cr
\mathcal{\leq }\int_{\cS^{2}\times \cS^{2}}\left\vert P_{\ell }(\cos d(
x_{1},x_{2}) )\right\vert ^{4}dx_{1}\left\{
\int_{\cS^{2}}\left\vert P_{\ell }(\cos d( x_{2},x_{3})
)\right\vert ^{2}dx_{2}\right\} \left\{ \int_{\cS^{2}}\left\vert P_{\ell
}(\cos d( x_{3},x_{4}) )\right\vert ^{4}dx_{4}\right\}
dx_{3}=\cr
=O\left( \frac{\log \ell }{\ell ^{2}}\right) \times O\left( \frac{1}{\ell }%
\right) \times O\left( \frac{\log \ell }{\ell ^{2}}\right) =O\left( \frac{%
\log ^{2}\ell }{\ell ^{5}}\right)
}$$
and likewise
$$\displaylines{
\mathcal{K}_{\ell }(6;3)=\int_{(\cS^{2})^4}\left\vert
P_{\ell }(\cos d( x_{1},x_{2}) )\right\vert ^{3}\left\vert
P_{\ell }(\cos d( x_{2},x_{3}) )\right\vert ^{3}\times \cr
\times \left\vert
P_{\ell }(\cos d( x_{3},x_{4}) )\right\vert ^{3}\left\vert
P_{\ell }(\cos d( x_{4},x_{1}) )\right\vert
^{3}dx_{1}dx_{2}dx_{3}dx_{4}\le \cr
\mathcal{\leq }\int_{(\cS^{2})^4}\left\vert P_{\ell
}(\cos d( x_{1},x_{2}) )\right\vert ^{3}\left\vert P_{\ell
}(\cos d( x_{2},x_{3}) )\right\vert ^{3}\left\vert P_{\ell
}(\cos d( x_{3},x_{4}) )\right\vert
^{3}dx_{1}dx_{2}dx_{3}dx_{4}=\cr
=O\left( \frac{\sqrt{\log \ell }}{\ell ^{3/2}}\right) \times O\left( \frac{%
\sqrt{\log \ell }}{\ell ^{3/2}}\right) \times O\left( \frac{\sqrt{\log \ell }%
}{\ell ^{3/2}}\right) =O\left( \frac{\log ^{3/2}\ell }{\ell ^{9/2}}\right)
\text{ .}
}$$
Finally for $q=7$
$$\displaylines{
\mathcal{K}_{\ell }(7;1) =\int_{\cS^{2}\times ...\times S^{2}}\left\vert
P_{\ell }(\cos d( x_{1},x_{2}) )\right\vert ^{6}\left\vert
P_{\ell }(\cos d( x_{2},x_{3}) )\right\vert \times \cr
\times \left\vert
P_{\ell }(\cos d( x_{3},x_{4}) )\right\vert ^{6}\left\vert
P_{\ell }(\cos d( x_{4},x_{1}) )\right\vert
dx_{1}dx_{2}dx_{3}dx_{4} \le \cr
\leq \int_{\cS^{2}\times S^{2}}\left\vert P_{\ell }(\cos d(
x_{1},x_{2}) )\right\vert ^{6}dx_{1}\left\{
\int_{\cS^{2}}\left\vert P_{\ell }(\cos d( x_{2},x_{3})
)\right\vert dx_{3}\right\} \left\{ \int_{\cS^{2}}\left\vert P_{\ell
}(\cos d( x_{3},x_{4}) )\right\vert ^{6}dx_{4}\right\}
dx_{2} =\cr
=O\left( \frac{1}{\ell ^{2}}\right) \times O\left( \frac{1}{\ell ^{1/2}}%
\right) \times O\left( \frac{1}{\ell ^{2}}\right) =O\left( \frac{1}{\ell
^{9/2}}\right)
}$$
and repeating the same argument we obtain
$$
\mathcal{K}_{\ell }(7;2)=O\left( \frac{1}{\ell ^{5}}%
\right)\qquad \text{and}\qquad \mathcal{K}_{\ell }(7;3) =
O\left( \frac{\log^{9/2} \ell}{\ell ^{11/2}}\right)\ .
$$
From
 \paref{simm}, we have indeed computed the bounds for $\mathcal{K}_{\ell
}(q;r)$, $q=1,\dots,7$ and $r=1,\dots, q-1$.

To conclude the proof  we note that, for $q>7$
\begin{equation*}
\max_{r=1,...,q-1}\mathcal{K}_{\ell }(q;r)=\max_{r=1,...,\left[ \frac{q}{2}%
\right] }\mathcal{K}_{\ell }(q;r)\leq \max_{r=1,...,3}\mathcal{K}_{\ell
}(6;r)=O\left( \frac{1}{\ell ^{9/2}}\right) \text{ .}
\end{equation*}%
Moreover in particular
\begin{equation*}
\max_{r=2,...,\left[ \frac{q}{2}\right] }\mathcal{K}_{\ell }(q;r)\leq
\mathcal{K}_{\ell }(7;2)\vee \mathcal{K}_{\ell }(7;3)=O\left( \frac{1}{\ell
^{5}}\right) \text{ ,}
\end{equation*}%
so that the dominant terms are of the form $\mathcal{K}_{\ell }(q;1).$
\end{proof}
We can now move to the higher-dimensional case, as follows. Let us start
with the bounds for all order moments of Gegenbauer polynomials.
From (\ref%
{momento 2})
\begin{equation}
\int_{{\mathbb{S}}^{d}}G_{\ell ;d}(\cos d( x_{1},x_{2}) )^{2}dx_{1}=O\left(
\frac{1}{\ell ^{d-1}}\right) \text{ ;}  \label{trid2}
\end{equation}%
also, from Proposition \ref{varianza}, we have that if $q=2p,$ $p=2,3,4...$,
\begin{equation}
\int_{{\mathbb{S}}^{d}}G_{\ell ;d}(\cos d( x_{1},x_{2}) )^{q}dx_{1}=O\left(
\frac{1}{\ell ^{d}}\right) \text{ .}  \label{tridq}
\end{equation}%
Finally, it is trivial to show that%
\begin{equation}
\int_{{\mathbb{S}}^{d}}\left\vert G_{\ell ;d}(\cos d( x_{2},x_{3})
)\right\vert dx_{2}\leq \sqrt{\int_{{\mathbb{S}}^{d}} G_{\ell
;d}(\cos d( x_{2},x_{3}) )^{2}dx_{2}}=O\left( \frac{1}{\sqrt{%
\ell ^{d-1}}}\right) \text{ ,}  \label{trid1}
\end{equation}%
\begin{equation}
\int_{{\mathbb{S}}^{d}}\left\vert G_{\ell ;d}(\cos d( x_{2},x_{3})
)\right\vert ^{3}dx_{2}\leq \sqrt{\int_{{\mathbb{S}}^{d}}G_{\ell
;d}(\cos d( x_{2},x_{3}) )^{2}dx_{2}}\sqrt{\int_{{\mathbb{S}}^{d}}G_{\ell
;d}(\cos d( x_{1},x_{2}) )^{4}dx_{1}}=O\left( \frac{1}{\ell ^{d-{\textstyle%
\frac{1}{2}}}}\right) \text{ }  \label{trid3}
\end{equation}%
and for $q\geq 5$ odd,
\begin{equation}
\int_{{\mathbb{S}}^{d}}\left\vert G_{\ell ;d}(\cos d( x_{2},x_{3})
)\right\vert ^{q}dx_{2}\leq \sqrt{\int_{{\mathbb{S}}^{d}}G_{\ell
;d}(\cos d( x_{2},x_{3}) )^{4}dx_{2}}\sqrt{\int_{{\mathbb{S}}^{d}}G_{\ell
;d}(\cos d( x_{1},x_{2}) )^{2(q-2)}dx_{1}}=O\left( \frac{1}{\ell ^{d}}%
\right) \text{ .}  \label{tridgen}
\end{equation}%
Analogously to  the $2$-dimensional case, we have the following.
\begin{prop}
\label{cumd} For all $r=1,2,...q-1,$
\begin{eqnarray}
 \mathcal{K}_{\ell }(q;r) &=&O\left( \frac{1}{\ell ^{2d+%
{\frac{d-5}{2}}}}\right) \text{ for }q=3\text{ ,}  \label{hoteld1} \\
\mathcal{K}_{\ell }(q;r) &=&O\left( \frac{1}{\ell ^{2d+%
{\frac{d-3}{2}}}}\right) \text{ for }q=4\text{ ,}  \label{hoteld2}
\end{eqnarray}%
and%
\begin{equation}
\mathcal{K}_{\ell }(q;1) = \mathcal{K}_{\ell
}(q;q-1) =O\left( \frac{1}{\ell ^{2d+{\frac{d-1}{2}}}}\right)
\text{ , }\mathcal{K}_{\ell }(q;r)=O\left( \frac{1}{%
\ell ^{3d-1}}\right) \text{ , }r=2,...,q-2,\text{ for }q\geq 5\text{ .}
\label{hoteld4}
\end{equation}
\end{prop}
\begin{proof} The proof relies on the same argument of the proof of Proposition \ref{cum2},
therefore we shall omit some calculations.
In what follows we exploit repeatedly the
inequalities \paref{tridq}, \paref{trid1}, \paref{trid3} and \paref{tridgen}.

For $q=3$ we immediately have
$$\displaylines{
\mathcal{K}_{\ell }(3;1)=\int_{(\cS^{d})^4}\left\vert G_{\ell ;d}(\cos d( x_{1},x_{2})
)\right\vert ^{2}\left\vert G_{\ell ;d}(\cos d(
x_{2},x_{3}) )\right\vert \times \cr
\times  \left\vert G_{\ell
;d}(\cos d( x_{3},x_{4}) )\right\vert
^{2}\left\vert G_{\ell ;d}(\cos d( x_{4},x_{1})
)\right\vert dx_{1}dx_{2}dx_{3}dx_{4}\le \cr
\mathcal{\leq }\int_{(\cS^{d})^4}\left\vert
G_{\ell ;d}(\cos d( x_{1},x_{2}) )\right\vert
^{2}\left\vert G_{\ell ;d}(\cos d( x_{2},x_{3})
)\right\vert \left\vert G_{\ell ;d}(\cos d(
x_{3},x_{4}) )\right\vert
^{2}dx_{1}dx_{2}dx_{3}dx_{4} =\cr
= O\left( \frac{1 }{\ell ^{d-1}}\right) \times O\left( \frac{1}{%
\sqrt{\ell^{d-1} }}\right) \times O\left( \frac{1 }{\ell ^{d-1}}\right)
=O\left( \frac{1 }{\ell ^{2d +\tfrac{d}{2} - \tfrac{5}{2}}}\right) \text{ .}
}$$
Likewise for $q=4$%
$$\displaylines{
\mathcal{K}_{\ell }(4;1)=\int_{(\cS^{d})^4}\left\vert G_{\ell ;d}(\cos d( x_{1},x_{2})
)\right\vert ^{3}\left\vert G_{\ell ;d}(\cos d(
x_{2},x_{3}) )\right\vert \times \cr
\times \left\vert G_{\ell
;d}(\cos d( x_{3},x_{4}) )\right\vert
^{3}\left\vert G_{\ell ;d}(\cos d( x_{4},x_{1})
)\right\vert dx_{1}dx_{2}dx_{3}dx_{4}\le \cr
\mathcal{\leq }\int_{(\cS^{d})^4}\left\vert
G_{\ell ;d}(\cos d( x_{1},x_{2}) )\right\vert
^{3}\left\vert G_{\ell ;d}(\cos d( x_{2},x_{3})
)\right\vert \left\vert G_{\ell
;d}(\cos d( x_{3},x_{4}) )\right\vert^3\,dx_{1}dx_{2}dx_{3}dx_{4} = \cr
= O\left( \frac{1 }{\ell ^{d-\tfrac12}}\right) \times
O\left(\frac{1 }{\ell ^{\tfrac{d}2-\tfrac12}}\right) \times
O\left( \frac{1 }{\ell ^{d-\tfrac12}}\right) =O\left( \frac{1
}{\ell ^{2d + \tfrac{d}2-\tfrac32}}\right)
}$$
and moreover
$$\displaylines{
\mathcal{K}_{\ell }(4;2)=\int_{(\cS^{d})^4}\left\vert G_{\ell ;d}(\cos d( x_{1},x_{2})
)\right\vert ^{2} \times \cr
\times \left\vert G_{\ell ;d}(\cos d(
x_{2},x_{3}) )\right\vert^2 \left\vert G_{\ell
;d}(\cos d( x_{3},x_{4}) )\right\vert
^{2}\left\vert G_{\ell ;d}(\cos d( x_{4},x_{1})
)\right\vert^2 dx_{1}dx_{2}dx_{3}dx_{4}\le \cr
\mathcal{\leq }\int_{(\cS^{d})^4}\left\vert
G_{\ell ;d}(\cos d( x_{1},x_{2}) )\right\vert
^{2}\left\vert G_{\ell ;d}(\cos d( x_{2},x_{3})
)\right\vert^2 \left\vert G_{\ell
;d}(\cos d( x_{3},x_{4}) )\right\vert^2\,dx_{1}dx_{2}dx_{3}dx_{4} =\cr
=O\left( \frac{1 }{\ell ^{d-1}}\right) \times O\left(
\frac{1 }{\ell ^{d-1}}\right) \times O\left( \frac{1 }{\ell
^{d-1}}\right) =O\left( \frac{1 }{\ell ^{3d - 3}}\right) \text{ .}
}$$
Similarly,
for $q=5$ we get the bounds
$$\displaylines{
\mathcal{K}_{\ell }(5;1)=\int_{\cS^{d}\times ...\times
S^{d}}\left\vert G_{\ell ;d}(\cos d( x_{1},x_{2})
)\right\vert ^{4} \times \cr
\times \left\vert G_{\ell ;d}(\cos d(
x_{2},x_{3}) )\right\vert \left\vert G_{\ell
;d}(\cos d( x_{3},x_{4}) )\right\vert
^{4}\left\vert G_{\ell ;d}(\cos d( x_{4},x_{1})
)\right\vert dx_{1}dx_{2}dx_{3}dx_{4}\le \cr
\mathcal{\leq }\int_{\cS^{d}\times ...\times S^{d}}\left\vert
G_{\ell ;d}(\cos d( x_{1},x_{2}) )\right\vert
^{4}\left\vert G_{\ell ;d}(\cos d( x_{2},x_{3})
)\right\vert \left\vert G_{\ell ;d}(\cos d(
x_{3},x_{4}) )\right\vert
^{4}dx_{1}dx_{2}dx_{3}dx_{4} =\cr
= O\left( \frac{1}{\ell ^{d}}\right) \times O\left( \frac{1}{%
\ell^{\tfrac{d}2-\tfrac12}}\right) \times O\left( \frac{1 }{\ell ^{d}}\right)
=O\left( \frac{ 1}{\ell ^{2d +\tfrac{d}2-\tfrac12 }}\right)
}$$
and
$$
\mathcal{K}_{\ell }(5;2)=O\left( \frac{1 }{\ell ^{3d -2}}\right)\ .
$$
It is immediate to check that
$$
\mathcal{K}_{\ell }(6;1)=\mathcal{K}_{\ell }(7;1)=O\left( \frac{1 }{\ell^{2d +\tfrac{d}2-\tfrac12 }}\right)\ ,\quad
\mathcal{K}_{\ell }(6;2)=\mathcal{K}_{\ell }(7;2)=O\left( \frac{1 }{\ell ^{2d + d -1}}\right)\ ,
$$
whereas
$$
\mathcal{K}_{\ell }(6;3)=O\left( \frac{1 }{\ell ^{2d + d - \tfrac32}}\right)\quad \text{and} \quad
\mathcal{K}_{\ell }(7;3)=O\left( \frac{1}{\ell ^{2d + d -\tfrac12}}\right)\ .
$$
The remaining terms are indeed bounded thanks to \paref{simm}.

In order to finish the proof, it is enough to note, as for  that for $q>7$%
\begin{equation}
\max_{r=1,...,q-1}\mathcal{K}_{\ell }(q;r)=\max_{r=1,...,\left[ \frac{q}{2}%
\right] }\mathcal{K}_{\ell }(q;r)\leq \max_{r=1,...,3}\mathcal{K}_{\ell
}(6;r)=O\left( \frac{1}{\ell^{2d +\tfrac{d}2-\tfrac12 }}\right) \text{ .}
\end{equation}%
In particular we have
\begin{equation}
\max_{r=2,...,\left[ \frac{q}{2}\right] }\mathcal{K}_{\ell }(q;r)\leq
\mathcal{K}_{\ell }(7;2)\vee \mathcal{K}_{\ell }(7;3)=O\left( \frac{1}{\ell
^{3d-1}}\right) \text{ ,}
\end{equation}%
so that the dominant terms are again of the form $\mathcal{K}_{\ell }(q;1).$
\end{proof}
Exploiting the results in this section and \S 3, we have the
following.
\begin{proof}[Proof Theorem \ref{teo1}]
For the case $q=2$ the standard CLT applies. For $q\ge 3$,
from Proposition \ref{BIGnourdinpeccati} and \paref{casoparticolare}, for $d_{\mathcal{D}}=
d_K, d_{TV}, d_W$
\begin{equation}
d_{\mathcal{D}}\left(\frac{h_{\ell ;q}}{\sqrt{\Var[h_{\ell ;q,d}]}},\mathcal{N}(0,1)\right)
 = O\left(\sup_{r}\sqrt{\frac{\mathcal{K}_{\ell }(q;r) }{%
 \Var[h_{\ell ;q,d}]^{2}}}\right)\text{ .}
\end{equation}%
The proof  is an immediate consequence of
the previous equality and the results in Proposition \ref{varianza},
 Proposition  \ref{cum2} and Proposition \ref{cumd}.
\end{proof}%
%

\section{General polynomials}

We show how the previous results can be extended to establish quantitative
CLTs, with no loss to the case of general, nonHermite
polynomials. To this aim, we need to introduce some more notation, namely
(for $Z_{\ell }$ defined as in (\ref{genovese}))%
\begin{equation*}
\mathcal{K}(Z_{\ell }):=\max_{q:\beta _{q}\neq 0}\max_{r=1,...,q-1}\mathcal{K%
}_{\ell }(q;r)\text{ ,}
\end{equation*}
\begin{equation*}
R(Z_{\ell })=\begin{cases}
\frac{1}{\ell ^{\frac{d-1}2}}\ ,\quad&\text{for }\beta _{2}\neq 0\ , \\
\max_{q=3,,,,Q:\beta _{q}\neq 0}R(\ell ;q,d)\ ,\quad&\text{for }\beta _{2}=0\ .%
\end{cases}
\end{equation*}
In words, $\mathcal{K}(Z_{\ell })$ is the largest contraction term among
those emerging from the analysis of the different Hermite components, and $%
R(Z_{\ell })$ is the slowest convergence rate of the same components. The
next result is stating that these are the only quantities to look at when
considering the general case.

%

\begin{proof}[Proof Theorem \ref{corollario1}]
We apply Proposition \ref{BIGnourdinpeccati}.
In our case%
$$\displaylines{
\Var[\langle  DZ_{\ell },-DL^{-1}Z_{\ell }\rangle _{H%
}]=\Var\left[ \langle \sum_{q_{1}=2}^{Q}\beta _{q_{1}}Dh_{\ell;q_{1},d},
-\sum_{q_{21}=2}^{Q}\beta _{q_{2}}DL^{-1}h_{\ell;q_{2},d}\rangle_{H}\right]=\cr
=\Var\left[ \sum_{q_{1}=2}^{Q}\sum_{q_{1}=2}^{Q}\beta _{q_{1}}\beta
_{q_{2}}\langle Dh_{\ell;q_{1},d},-DL^{-1}h_{\ell;q_{2},d}\rangle_{%
H}\right] \text{ .}
}$$
From \S $2.1$ recall that for $q_{1}\neq q_{2}$
\begin{equation*}
E[\langle Dh_{\ell;q_{1},d},-DL^{-1}h_{\ell ;q_{2},d}\rangle_{%
H}]=0\text{ ,}
\end{equation*}%
whence we write
$$\displaylines{
\Var\left[ \sum_{q_{1}=2}^{Q}\sum_{q_{2}=2}^{Q}\beta _{q_{1}}\beta
_{q_{2}}\langle Dh_{\ell;q_{1},d},-DL^{-1}h_{\ell;q_{2},d}\rangle_{%
H}\right]=\cr
=\sum_{q_{1}=2}^{Q}\sum_{q_{2}=2}^{Q}\beta _{q_{1}}^{2}\beta
_{q_{2}}^{2}\Cov\left( \langle Dh_{\ell;q_{1},d},-DL^{-1}h_{\ell;q_{1},d}
\rangle_{H},\langle Dh_{\ell;q_{2},d},-DL^{-1}h_{\ell ;q_{2},d}\rangle_{H}\right)+\cr
+\sum_{q_{1}=2}^{Q}\sum_{q_{2}\neq
q_{1}}^{Q}\sum_{q_{3}=2}^{Q}\sum_{q_{4}\neq q_{3}}^{Q}\beta _{q_{1}}\beta
_{q_{2}}\beta _{q_{3}}\beta _{q_{4}}
\Cov\left( \langle Dh_{\ell;q_{1},d},-DL^{-1}h_{\ell;q_{2},d}\rangle_{H},\langle
Dh_{\ell ;q_{3},d},-DL^{-1}h_{\ell ;q_{4},d}\rangle_{H}\right)
\text{ .}
}$$
Now of course we have
$$\displaylines{
\Cov\left( \langle Dh_{\ell ;q_{1},d},-DL^{-1}h_{\ell
;q_{1},d}\rangle_{H},\langle Dh_{\ell
;q_{2},d},-DL^{-1}h_{\ell ;q_{2},d}) _{H}\right) \le \cr
\leq \left(\Var\left[\langle Dh_{\ell ;q_{1},d},-DL^{-1}h_{\ell
;q_{1},d}\rangle_{H}\right] \Var\left[ \langle Dh_{\ell
;q_{2},d},-DL^{-1}h_{\ell ;q_{2},d}\rangle_{H}\right] \right)^{1/2},
}$$
$$\displaylines{
\Cov\left( \langle Dh_{\ell ;q_{1},d},-DL^{-1}h_{\ell
;q_{2},d}\rangle_{H},\langle Dh_{\ell
;q_{3},d},-DL^{-1}h_{\ell ;q_{4},d}\rangle_{H}\right) \le \cr
\leq \left( \Var\left[ \langle Dh_{\ell ;q_{1},d},-DL^{-1}h_{\ell
;q_{2},d}\rangle_{H}\right] \Var\left[ \langle Dh_{\ell
;q_{3},d},-DL^{-1}h_{\ell ;q_{4},d}\rangle_{H}\right]
\right)^{1/2}.
}$$
Applying \cite{noupebook}, Lemma 6.2.1 it is immediate to show that%
$$\displaylines{
\Var\left[ \langle Dh_{\ell ;q_{1},d},-DL^{-1}h_{\ell ;q_{1},d}\rangle
_{H}\right]\le \cr
\leq q_{1}^{2}\sum_{r=1}^{q_{1}-1}((r-1)!)^{2}{
q_{1}-1 \choose
r-1%
}^{4}(2q_{1}-2r)!\left\Vert g_{\ell; q_{1}}\otimes
_{r}g_{\ell;q_{1}}\right\Vert _{H^{\otimes 2q_1-2r}}^{2} =\cr
=q_{1}^{2}\sum_{r=1}^{q_{1}-1}((r-1)!)^{2}{
q_{1}-1 \choose
r-1%
} ^{4}(2q_{1}-2r)! \mathcal{K}_{\ell }(q_1;r) \text{
.}
}$$
Also, for $q_{1}<q_{2}$%
$$\displaylines{
\Var\left[ \langle Dh_{\ell ;q_{1},d},-DL^{-1}h_{\ell ;q_{2},d}\rangle
_{H}\right]=\cr
=q_{1}^{2}\sum_{r=1}^{q_{1}}((r-1)!)^{2}{
q_{1}-1 \choose
r-1%
} ^{2}{
q_{2}-1 \choose
r-1%
} ^{2}(q_{1}+q_{2}-2r)!\left\Vert g_{\ell;q_{1}}\widetilde{\otimes }%
_{r}g_{\ell ; q_{2}}\right\Vert _{H^{\otimes
(q_{1}+q_{2}-2r)}}^{2}=\cr
=q_{1}^{2}((q_{1}-1)!)^{2}{
q_{2}-1 \choose
q_{1}-1%
}^{2}(2q_{1}-2r)!\left\Vert g_{\ell;q_{1}}\widetilde{\otimes }%
_{q_{1}}g_{\ell ; q_{2}}\right\Vert _{H^{\otimes (q_{2}-q_1)}}^{2} + \cr
+q_{1}^{2}\sum_{r=1}^{q_{1}-1}((r-1)!)^{2}{
q_{1}-1 \choose
r-1%
}^{2}{
q_{2}-1 \choose
r-1%
} ^{2}(q_{1}+q_{2}-2r)!\left\Vert g_{\ell;q_{1} }\widetilde{\otimes }%
_{r} g_{\ell ;q_{2}}\right\Vert _{H^{\otimes
(q_{1}+q_{2}-2r)}}^{2}
=:A+B\text{ .}
}$$%
Let us focus on the first summand $A$, which includes terms that, from Lemma \ref{contractions},
take the form%
$$\displaylines{
\left\Vert g_{\ell;q_{1}}\widetilde{\otimes }_{q_{1}}g_{\ell;q_{2}
}\right\Vert_{H^{\otimes (q_{2}-q_1)}}^{2}\le
\left\Vert g_{\ell;q_{1}}\otimes_{q_{1}}g_{\ell;q_{2}
}\right\Vert_{H^{\otimes (q_{2}-q_1)}}^{2}=\cr
=\int_{(\cS^d)^{q_2-q_1}}\int_{(\cS^{d})^2}\left( \frac{n_{\ell; d}%
}{\mu_d }\right)^{q_{2}-q_{1}}
 G_{\ell;d }(\cos d( x_{2},y_{1}) )...G_{\ell;d }(\cos d( x_{2},y_{q_{2}-q_1}) )
G_{\ell;d }(\cos d(
x_{1},x_2) )^{q_1}\,dx_1 dx_2\times\cr
\times \int_{(\cS^d)^2}G_{\ell;d}(\cos d( x_{3},y_{1}) )
...G_{\ell;d}(\cos d(
x_{3},y_{q_{2}-q_1}) )G_{\ell;d }(\cos d(
x_{3},x_4) )^{q_1}\,dx_3 dx_4\,d\underline{y} =: I\ ,
}$$
where for the sake of simplicity we have set $d\underline{y}:=dy_{1}...dy_{q_2-q_1}$.
Applying $q_2-q_1$ times the reproducing formula for Gegenbauer polynomials (\cite{szego})  we get
\begin{equation}\label{anvedi}
I = \int_{(\cS^{d})^4}G_{\ell;d}(\cos d(
x_{1},x_{2}) )^{q_{1}}G_{\ell;d}(\cos d(
x_{2},x_{3}) )^{q_{2}-q_{1}}G_{\ell;d}(\cos d(
x_{3},x_{4}) )^{q_{1}}\,d\underline{x}\ .
\end{equation}
In graphical terms, these contractions correspond to the diagrams such that
all $q_{1}$ edges corresponding to vertex $1$ are linked to vertex 2, vertex
$2$ and $3$ are connected by $q_{2}-q_{1}$ edges, vertex $3$ and $4$ by $%
q_{1}$ edges, and no edges exist between $1$ and $4,$ i.e. the diagram has
no proper loop.

Now immediately we write
$$\displaylines{
\paref{anvedi} =\int_{\cS^{d}}G_{\ell;d}(\cos d( x_{1},x_{2})
)^{q_{1}}\,dx_{1}\int_{\cS^{d}}G_{\ell;d}(\cos d( x_{3},x_{4})
)^{q_{1}}\,dx_{4}\int_{(\cS^{d})^2}G_{\ell;d}(\cos d(
x_{2},x_{3}) )^{q_{2}-q_{1}}\,dx_{2}dx_{3}=\cr
=\frac{1}{(q_{1}!)^{2}} \Var[ h_{\ell;q_1,d}]^{2}
\int_{(\cS^{d})^2}G_{\ell;d}
(\cos d( x_{2},x_{3}) )^{q_{2}-q_{1}}\,dx_{2}dx_{3}\text{ .}
}$$
Moreover we have
\begin{equation}\label{eq=0}
\int_{(\cS^{d})^2}G_{\ell;d}(\cos d(
x_{2},x_{3}) )^{q_{2}-q_{1}}\,dx_{2}dx_{3}=0\ ,\quad \text{if}\ q_{2}-q_{1}=1\
\end{equation}%
and from \paref{momento 2} if $q_{2}-q_{1}\geq 2$
\begin{equation*}
\int_{(\cS^{d})^2}G_{\ell;d}(\cos d(
x_{2},x_{3}) )^{q_{2}-q_{1}}\,dx_{2}dx_{3}\leq \mu_d \int_{\cS^{d}}G_{\ell;d}(\cos d(x,y))^{2}\,dx=
O\left(\frac{1}{\ell^{d-1} }\right)\ .
\end{equation*}%
It follows that
\begin{equation}
\left\Vert g_{\ell;q_{1} }\otimes_{q_{1}}g_{\ell;q_{2}
}\right\Vert_{H^{\otimes (q_2-q_1)}}^{2}=  O\left(\Var[ h_{\ell;q_1,d}]^{2}\frac{1}{\ell^{d-1} }\right)
\label{efficientbound}
\end{equation}%
always. For the second term, still from \cite{noupebook}, Lemma $6.2.1$ we have%
\begin{eqnarray*}
B \leq \frac{q_{1}^{2}}{2}\sum_{r=1}^{q_{1}-1}((r-1)!)^{2}\left(
\begin{array}{c}
q_{1}-1 \\
r-1%
\end{array}%
\right) ^{2}\left(
\begin{array}{c}
q_{2}-1 \\
r-1%
\end{array}%
\right) ^{2}(q_{1}+q_{2}-2r)!\times \\
\times \left( \left\Vert g_{\ell;q_{1} }\otimes _{q_{1}-r}g_{\ell;q_{1}
}\right\Vert _{H^{\otimes 2r}}^{2}+\left\Vert g_{\ell;q_{2}
}\otimes _{q_{2}-r}g_{\ell;q_{2} }\right\Vert _{H^{\otimes
2r}}^{2}\right)=
\end{eqnarray*}%
\begin{equation}
=\frac{q_{1}^{2}}{2}\sum_{r=1}^{q_{1}-1}((r-1)!)^{2}\left(
\begin{array}{c}
q_{1}-1 \\
r-1%
\end{array}%
\right) ^{2}\left(
\begin{array}{c}
q_{2}-1 \\
r-1%
\end{array}%
\right) ^{2}(q_{1}+q_{2}-2r)!\left( \mathcal{K}_{\ell }(q_{1};r)+\mathcal{K}%
_{\ell }(q_{2};r)\right)\ , \label{juve2}
\end{equation}%
where the last step follows from Lemma \ref{contractions}.

Let us first investigate the case $d=2$. From  \paref{q=2}, \paref{int2} and \paref{q=4d=2}
it is immediate that
\begin{equation}
\Var[Z_{\ell }]=\sum_{q=2}^{Q}\beta _{q}^{2}\Var[h_{\ell ;q}]=
\begin{cases}
O(\ell ^{-1})\ ,\quad &\text{for }\beta _{2}\neq 0 \\
O(\ell ^{-2}\log \ell)\ ,\quad &\text{for }\beta _{2}=0\text{ , }\beta _{4}\neq 0 \\
O(\ell ^{-2})\ ,\quad &\text{otherwise.}%
\end{cases}%
\end{equation}%
Hence we have that for $\beta _{2}\neq 0$
\begin{eqnarray*}
d_{TV}\left(\frac{Z_{\ell }-EZ_{\ell }}{\sqrt{\Var[Z_{\ell }]}},\mathcal{N}(0,1)\right)
=O\left(\frac{\sqrt{\mathcal{K}_{\ell }(2;r)}}{\Var[Z_{\ell }]}\right)
=O\left(\ell ^{-1/2}\right)\text{ ;}
\end{eqnarray*}%
for $\beta _{2}=0$ , $\beta _{4}\neq 0$ ,%
\begin{eqnarray*}
d_{TV}(\frac{Z_{\ell }-EZ_{\ell }}{\sqrt{\Var[Z_{\ell }]}},\mathcal{N}(0,1))
=O(\frac{\sqrt{\mathcal{K}_{\ell }(4;r)}}{\Var[Z_{\ell }]})
=O\left(\frac{1}{\log \ell }\right)
\end{eqnarray*}%
and for $\beta _{2}=\beta _{4}=0$, $\beta _{5}\neq 0$ and $c_5 >0$
\begin{eqnarray*}
d_{TV}\left(\frac{Z_{\ell }-EZ_{\ell }}{\sqrt{\Var[Z_{\ell }]}},\mathcal{N}(0,1)\right)
=O\left(\frac{\sqrt{\mathcal{K}_{\ell }(5;r)}}{\Var[Z_{\ell }]}\right)
=O\left(\frac{\log \ell }{\ell ^{1/4}}\right)\text{ .}
\end{eqnarray*}%
and analogously we deal with the remaining cases, so that we obtain the claimed result for $d=2$.

For $d\ge 3$ from \paref{momento 2} and Proposition \ref{varianza}, it holds
\begin{equation*}
\Var[Z_{\ell }]=\sum_{q=2}^{Q}\beta _{q}^{2}\Var(h_{\ell ;q,d})=
\begin{cases}
O(\ell ^{-d+1})\ ,\quad &\text{for }\beta _{2}\neq 0\ , \\
O(\ell ^{-d})\ ,\quad &\text{otherwise}\ .
\end{cases}%
\end{equation*}
Hence  we have for $\beta _{2}\neq 0$
\begin{eqnarray*}
d_{TV}\left(\frac{Z_{\ell }-EZ_{\ell }}{\sqrt{\Var[Z_{\ell }]}},\mathcal{N}(0,1)\right)
=O\left(\frac{\sqrt{\mathcal{K}_{\ell }(2;r)}}{\Var[Z_{\ell }]}\right) =
O\left(\frac{1}{\ell ^{\frac{d-1}{2}}}\right)\text{ .}
\end{eqnarray*}%
Likewise for $\beta _{2}=0$ , $\beta _{3},c_{3;d}\neq 0$,
\begin{eqnarray*}
d_{TV}\left(\frac{Z_{\ell }-EZ_{\ell }}{\sqrt{\Var[Z_{\ell }]}},\mathcal{N}(0,1)\right)
=O\left(\frac{\sqrt{\mathcal{K}_{\ell }(3;r)}}{\Var[Z_{\ell }]}\right)
=O\left(\frac{1}{\ell ^\frac{d-5}{4}}\right)\text{ %
}
\end{eqnarray*}%
and for $\beta _{2}=\beta _{3}=0$, $\beta _{4}\neq 0$
\begin{eqnarray*}
d_{TV}\left(\frac{Z_{\ell }-EZ_{\ell }}{\sqrt{\Var[Z_{\ell }]}},\mathcal{N}(0,1)\right)
=O\left(\frac{\sqrt{\mathcal{K}_{\ell }(4;r)}}{\Var[Z_{\ell }]}\right)
=O\left(\frac{1}{\ell ^{\frac{d-3}{2}}}\right)\text{ .%
}
\end{eqnarray*}%
Finally if $\beta_2=\beta_3=\beta_4=0$, $\beta_q, c_{q;d} \ne 0$ for some $q$, then
\begin{eqnarray*}
d_{TV}\left(\frac{Z_{\ell }-EZ_{\ell }}{\sqrt{\Var[Z_{\ell }]}},\mathcal{N}(0,1)\right)
=O\left(\frac{\sqrt{\mathcal{K}_{\ell }(q;r)}}{\Var[Z_{\ell }]}\right)
=O\left(\sqrt{\frac{\ell^{2d} }{\ell ^{2d +\frac{d}{2}-\frac{1}{2}}}}\right)=
O\left(\frac{1}{\ell ^{\frac{d-1}{4}}}\right)\text{ .}
\end{eqnarray*}
\end{proof}
\begin{remark}
\label{rem0}\textrm{To compare our result in these specific circumstances
with the general bound obtained by Nourdin and Peccati, we note that for \paref{anvedi},
these
authors are exploiting the inequality%
\begin{equation*}
\displaylines{ \left\Vert g_{\ell;q_{1}}\otimes_{q_{1}} g_{\ell;q_{2}}
\right\Vert_{H^{\otimes (q_2-q_1)}} ^{2}\le \left\Vert g_{\ell;q_{1}}
\right\Vert _{H^{\otimes q_{1}}}^{2}\left\Vert g_{\ell;q_{2}}\otimes_{q_{2}-q_{1}}g_{\ell;q_{2}}
\right\Vert_{H^{\otimes 2q_{1}}}\ , }
\end{equation*}%
see \cite{noupebook}, Lemma $6.2.1$. In the special framework we
consider here (i.e., orthogonal eigenfunctions), this provides, however,  a less efficient bound than
(\ref{efficientbound}): indeed from \paref{anvedi}, repeating the same argument as in Lemma
\ref{contractions},
one obtains
$$\displaylines{
\left\Vert g_{\ell ;q_{1}}\otimes _{q_{1}}g_{\ell ;q_{2}}
\right\Vert_{H^{\otimes (q_2-q_1)}}^{2}=
\int_{({\mathbb{S}}^{d})^4}G_{\ell ;d}(\cos
d(x_{1},x_{2}))^{q_1}G_{\ell ;d}(\cos d(x_{2},x_{3}))^{q_2-q_1}G_{\ell ;d}(\cos
d(x_{3},x_{4}))^{q_1}\,d\underline{x}\le \cr
 \le \int_{({\mathbb{S}}^{d})^2} G_{\ell
;d}(\cos d(x_{1},x_{2}))^{q_1}\,dx_1 dx_2\times \cr
\times \sqrt{\int_{({\mathbb{S}}^{d})^4}G_{\ell ;d}(\cos d(x_{1},x_{2}))^{q_1}
G_{\ell ;d}(\cos
d(x_{2},x_{3}))^{q_2-q_1}G_{\ell ;d}(\cos
d(x_{3},x_{4}))^{q_1}G_{\ell ;d}(\cos d(x_{1},x_{4}))^{q_2-q_1}\,d\underline{x}}=\cr
=O\left(  \Var[h_{\ell ;q_{1},d}] \sqrt{\mathcal{K}_{\ell
}(q_{2},q_{1})}\right) \text{ ,}
}$$
yielding a bound of order
\begin{equation}
O\left( \sqrt{\frac{ \Var[h_{\ell;q_{1},d}] \sqrt{\mathcal{K}%
_{\ell }(q_{2},q_{1})}}{\Var[h_{\ell ;q_{1},d}]^{2}}}\right)
=O\left( \frac{\sqrt[4]{\mathcal{K}_{\ell }(q_{2},q_{1})}}{\sqrt{\Var[h_{\ell
;q_{1},d}]}}\right)  \label{cdip}
\end{equation}%
rather than%
\begin{equation}
O\left( \sqrt{\frac{\mathcal{K}_{\ell }(q_{2},q_{1})}{ \Var[h_{\ell
;q_{1},d}]^{2}}}\right) \text{ ;}  \label{cdip2}
\end{equation}%
for instance, for $d=2$ note that (\ref{cdip}) is typically $=O(\ell
\times \ell ^{-9/8})=O(\ell ^{-1/8}),$ while we have established for (\ref{cdip2}%
) bounds of order $O(\ell ^{-1/4})$. }
\end{remark}

\begin{remark}
\textrm{Clearly the fact that $\left\Vert g_{\ell;q_{1}
}\otimes_{q_{1}}g_{\ell;q_{2} }\right\Vert_{H^{\otimes (q_2-q_1)}}^{2}=0$ for $q_{2}=q_{1}+1$
entails that the contraction $g_{\ell;q_{1} }\otimes_{q_{1}}g_{\ell;q_{2} }$ is identically
null. Indeed repeating the same argument as in Lemma \ref{contractions}
$$\displaylines{
g_{\ell;q_{1} }\otimes_{q_{1}}g_{\ell; q_{1}+1}=\cr
=
\int_{(\cS^{d})^{2}}G_{\ell ;d}(\cos d( x_{1},y) )G_{\ell ;d}(\cos d(
x_{1},x_2) )^{q_1}\,dx_{1}dx_{2}=\cr
=\int_{{\mathbb{S}}^{d}}G_{\ell ;d}(\cos d( x_{1},y) )\,dx_{1}\int_{{%
\mathbb{S}}^{d}}G_{\ell ;d}(\cos d( x_{1},x_2) )^{q_1}\,dx_{2}=0\text{ ,}
}$$
as expected. }
\end{remark}

\section{General nonlinear functionals and excursion sets}

The techniques and results developed in \S $4,5$ are restricted to
finite-order polynomials. \ In the special case of the Wasserstein distance,
we shall show below how they can indeed be extended to general nonlinear
functionals of the form \paref{S}
\begin{equation*}
S_{\ell }(M)=\int_{{\mathbb{S}}^{d}}M(T_{\ell }(x))dx\text{ ;}
\end{equation*}%
here $M:\mathbb{R}\rightarrow \mathbb{R}$ is a measurable function such that
$\mathbb{E}[M(T_{\ell })^{2}]<\infty $ and $J_{2}(M)\neq 0,$ where we recall
that $J_{q}(M):=\mathbb{E}[M(T_{\ell })H_{q}(T_{\ell })]$ .

\begin{remark}
\textrm{Without loss of generality, the first two coefficients $%
J_{0}(M),J_{1}(M)$ can always be taken to be zero in the present framework.
Indeed, $J_{0}(M):=\mathbb{E}[M(T_{\ell })]=0,$ assuming we work with
centred variables and moreover as we noted earlier $h_{\ell ;1,d}=\int_{{%
\mathbb{S}}^{d}}T_{\ell }(x)\,dx=0$. }
\end{remark}
\begin{proof}[Proof Theorem \ref{general}]
As in \cite{MaWi}, from \paref{exp} we write the expansion%
\begin{equation*}
S_{\ell }(M) =\int_{\cS^{d}}\sum_{q=2}^{\infty }\frac{J_{q}(M)H_{q}(T_{\ell
}(x))}{q!}dx\ .
\end{equation*}
Precisely, we write for $d=2$
\begin{eqnarray}\label{sum2}
S_{\ell }(M) =\frac{J_{2}(M)}{2%
}h_{\ell;2,2}+ \frac{J_{3}(M)}{3!}h_{\ell;3,2} + \frac{J_{4}(M)}{4!}h_{\ell;4,2}
+\int_{\cS^{2}}\sum_{q=5}^{\infty }\frac{J_{q}(M)H_{q}(T_{\ell }(x))}{q!%
}dx\text{ ,}
\end{eqnarray}%
whereas for $d\ge 3$
\begin{eqnarray}\label{sum2d}
S_{\ell }(M) =\frac{J_{2}(M)}{2}h_{\ell;2,d}
+\int_{\cS^{d}}\sum_{q=3}^{\infty }\frac{J_{q}(M)H_{q}(T_{\ell }(x))}{q!%
}dx\text{ .}
\end{eqnarray}%
Let us first investigate the case $d=2$.
Set for the sake of simplicity%
\[
S_{\ell }(M;1):=\frac{J_{2}(M)}{2%
}h_{\ell;2,2}+ \frac{J_{3}(M)}{3!}h_{\ell;3,2} + \frac{J_{4}(M)}{4!}h_{\ell;4,2}\text{ ,}
\]%
\[
S_{\ell }(M;2):=\int_{\cS^{2}}\sum_{q=5}^{\infty }\frac{J_{q}(M)H_{q}(T_{\ell
}(x))}{q!}dx\text{ .}
\]%
Hence from \paref{sum2} and the triangular inequality
$$\displaylines{
d_{W}\left( \frac{S_{\ell }(M)}{\sqrt{\Var[S_{\ell }(M)]}},\mathcal{N}%
(0,1)\right)\le\cr
\le d_{W}\left( \frac{S_{\ell }(M)}{\sqrt{\Var[S_{\ell }(M)]}},\frac{%
S_{\ell }(M;1)}{\sqrt{\Var[S_{\ell }(M)]}}\right) +d_{W}\left( \frac{S_{\ell
}(M;1)}{\sqrt{\Var[S_{\ell }(M)]}},\mathcal{N}\left(0,\frac{\Var[S_{\ell }(M;1)]}{%
\Var[S_{\ell }(M)]}\right)\right)+ \cr
+d_{W}\left( \mathcal{N}\left(0,\frac{\Var[S_{\ell }(M;1)]}{\Var[S_{\ell }(M)]}\right),%
\mathcal{N}(0,1)\right)\le \cr
\le \frac{1}{\sqrt{\Var[S_{\ell }(M)]}}\mathbb{E}\left[\left(
\int_{\cS^{2}}\sum_{q=5}^{\infty }\frac{J_{q}(M)H_{q}(T_{\ell }(x))}{q!}dx%
\right)^2\right]^{1/2}+ \cr
+d_{W}\left( \frac{S_{\ell }(M;1)}{\sqrt{\Var[S_{\ell }(M)]}},\mathcal{N}(0,%
\frac{\Var[S_{\ell }(M;1)]}{\Var[S_{\ell }(M)]})\right) +d_{W}\left( \mathcal{N%
}(0,\frac{\Var[S_{\ell }(M;1)]}{\Var[S_{\ell }(M)]}),\mathcal{N}(0,1)\right)
\text{ .}
}$$
Let us bound the first term of the previous summation. Of course
$$\displaylines{
\Var[S_{\ell }(M)] = \Var[S_{\ell }(M;1)] +
\Var[S_{\ell }(M;2)]\ ;
}$$
now we have (see \cite{MaWi})
$$
\Var[S_{\ell }(M;1)] =\frac{J_{2}^{2}(M)}{2}\Var[h_{\ell ;2,2}]+\frac{%
J_{3}^{2}(M)}{6}\Var[h_{\ell ;3,2}]
+\frac{J_{4}^{2}(M)}{4!}\Var[h_{\ell ;4,2}]
$$
and moreover
$$\displaylines{
\Var[S_{\ell }(M;2)]=
\mathbb{E}\left[ \left(\int_{\cS^{2}}\sum_{q=5}^{\infty }\frac{J_{q}(M)H_{q}(T_{\ell
}(x))}{q!}dx \right)^2\right ]=\sum_{q=5}^{\infty }\frac{J_{q}^{2}(M)}{q!}%
\Var[h_{\ell ;q,2}]
\ll \frac{1}{\ell ^{2}}\sum_{q=5}^{\infty }\frac{J_{q}^{2}(M)}{q!}\ll
\frac{1}{\ell ^{2}}\text{ ,}
}$$
where the last bounds follows from \paref{int2} and \paref{cq2}.
Therefore recalling also \paref{q=2} and \paref{q=4d=2}
\[
\frac{1}{\Var[S_{\ell }(M)]}\mathbb{E}\left[ \left(\int_{\cS^{2}}\sum_{q=5}^{\infty }%
\frac{J_{q}(M)H_{q}(T_{\ell }(x))}{q!}dx\right)^2\right]\ll \frac{1}{\ell }%
\text{ .}
\]%
On the other hand, from Theorem \ref{corollario1}%
\[
d_{W}\left( \frac{S_{\ell }(M;1)}{\sqrt{\Var[S_{\ell }(M)]}},\mathcal{N}\left(0,%
\frac{\Var[S_{\ell }(M;1)]}{\Var[S_{\ell }(M)]}\right)\right) =O\left(\frac{1}{\sqrt{\ell
}}\right)
\]%
and finally, using Proposition 3.6.1 in \cite{noupebook},
\begin{eqnarray*}
d_{W}\left( \mathcal{N}\left(0,\frac{\Var[S_{\ell }(M;1)]}{\Var[S_{\ell }(M)]}\right),%
\mathcal{N}(0,1)\right) &\leq &\sqrt{\frac{2}{\pi }}\left|\frac{%
\Var[S_{\ell }(M;1)]}{\Var[S_{\ell }(M)]}-1\right|=O\left(\frac{1}{\ell }\right)\text{ ,}
\end{eqnarray*}%
so that the proof for $d=2$ is completed.

The proof in the general case $d\ge 3$ is indeed
analogous, just setting
\[
S_{\ell }(M;1):=\frac{J_{2}(M)}{2%
}h_{\ell;2,d}\text{ ,}
\]%
\[
S_{\ell }(M;2):=\int_{\cS^{2}}\sum_{q=3}^{\infty }\frac{J_{q}(M)H_{q}(T_{\ell
}(x))}{q!}dx
\]%
and recalling from \paref{momento 2} that $\Var[h_{\ell;2,d}]=O( \frac{1}{\ell^{d-1}})$ whereas for
$q\ge 3$, $\Var[h_{\ell;q,d}]=O( \frac{1}{\ell^{d}})$ from Proposition \ref{varianza}.
\end{proof}
A remarkable special case is obtained for the excursion sets,
which for any fixed $z\in \mathbb{R}$ can be defined as%
\begin{equation*}
S_{\ell }(z):=S_{\ell }(\mathbb{I}(\cdot\le z))=\int_{{\mathbb{S}}^{d}}\mathbb{I%
}(T_{\ell }(x)\le z)dx\text{ ,}
\end{equation*}%
where $\mathbb{I}(\cdot \le z)$ is the indicator function of the interval $%
(-\infty,z]$. Note that ${\mathbb{E}}[S_{\ell }(z)]=\mu _{d}\Phi (z)$, where
$\Phi (z)$ is the cdf of the standard Gaussian law, and in this case we have
$M=M_z:=\mathbb{I}(\cdot \le z)$, $J_{2}(M_z)=z\phi (z)$, $\phi$ denoting
 the standard Gaussian density.
 The following corollary
is then immediate:

\begin{cor}
If $z\ne 0$, as $\ell \rightarrow \infty ,$ we have that%
\begin{equation*}
d_{W}\left( \frac{S_{\ell }(z)-\mu_d \Phi(z)}{\sqrt{\mathrm{Var}[S_{\ell
}(z)]}}, \mathcal{N}(0,1)\right) =O\left(\frac{1}{\sqrt{\ell }}\right)\text{
.}
\end{equation*}
\end{cor}

\begin{remark}
\textrm{It should be noted that the rate obtained here is much sharper than
the one provided by \cite{pham} for the Euclidean case with $d=2$. The
asymptotic setting we consider is rather different from his, in that we
consider the case of spherical eigenfunction with diverging eigenvalues,
whereas he focusses on functionals evaluated on increasing domains $%
[0,T]^{d} $ for $T\rightarrow \infty .$ However the contrast in the
converging rates is not due to these different settings, indeed \cite{cammar}
establish rates of convergence analogous to those by \cite{pham} for
spherical random fields with more rapidly decaying covariance structure than
the one we are considering here. The main point to notice is that the slow
decay of Gegenbauer polynomials entails some form of long range dependent
behaviour on random spherical harmonics; in this sense, hence, our results
may be closer in spirit to the work by \cite{dehlingtaqqu} on empirical
processes for long range dependent stationary processes on $\mathbb{R}$. }
\end{remark}

Department of Mathematics, University of Rome Tor Vergata, Via della Ricerca
Scientifica, 00133 Roma, Italy \smallskip

\texttt{marinucc@mat.uniroma2.it} \smallskip

\texttt{rossim@mat.uniroma2.it}

\end{document}